\documentclass[12pt,a4paper,oneside,reqno]{article}

\usepackage[utf8]{inputenc}
\usepackage[english]{babel}
\usepackage[left=1in,right=1in,top=1in,bottom=1in]{geometry}
\usepackage{amsfonts,amssymb,amsmath,amsthm}
\usepackage[symbol]{footmisc}
\usepackage{longtable}
\usepackage{enumerate}
\usepackage{cite}
\usepackage{xcolor}
\usepackage{hyperref}
\usepackage{fancyhdr}
\theoremstyle{plain}
\newtheorem{theorem}{Theorem}

\newtheorem{proposition}[theorem]{Proposition}
\newtheorem{remark}[theorem]{Remark}
\newtheorem{example}[theorem]{Example}

\newtheorem{corollary}[theorem]{Corollary}
\newtheorem{question}{Question}
\newtheorem{conjecture}{Conjecture}

\DeclareMathOperator{\Aut}{Aut}

\DeclareMathOperator{\Der}{Der}
\DeclareMathOperator{\Ann}{Ann}
\DeclareMathOperator{\GL}{GL}

\newcommand{\kk}{\Bbbk}
\newcommand{\AutS}{\mathbf{Aut}}
\newcommand{\Evo}{\mathsf{Evo}}
\newcommand{\src}[1]{\par\smallskip\noindent{\small\textit{Source and status:} #1}}

\usepackage{fouriernc}

\begin{document}

\bigskip

\noindent{\large
Evolution algebras: a survey of results and a list of open problems}\footnote{
The work is supported by
FCT 2023.08031.CEECIND and UID/00212/2025; by the Spanish Ministerio de Ciencia e
Innovaci\'on 10.13039/501100011033, grants PID2023-152673NB-I00 and PID2023-149804NB-I00;
by grant Fortalece 2023/03 of ``Comunidad Aut\'onoma de La Rioja''; and by the Junta de
Andaluc\'ia through project PPRO-FQM336-G-2023 (FQM336-G-FEDER).}
\footnote{The original concept, the initial list of problems and the overall framework were developed by
I.K., C.M.G.\ and A.V. The AI system Aristotle (Harmonic, \url{https://aristotle.harmonic.fun/}) assisted
the authors in the literature search, error verification, computer experiments and the formal
verification in Lean of the new observations of Section~\ref{sec:new}.}
\bigskip

\begin{center}
{\bf
Ivan Kaygorodov\footnote{CMA-UBI, University of Beira Interior, Covilh\~{a}, Portugal; \ kaygorodov.ivan@gmail.com}, \
C\'andido Mart\'in Gonz\'alez\footnote{Departamento de \'Algebra, Geometr\'ia y Topolog\'ia, Universidad de M\'alaga, M\'alaga, Spain; \ candido\_m@uma.es} \ \& \
Antonio Viruel\footnote{Departamento de \'Algebra, Geometr\'ia y Topolog\'ia, Universidad de M\'alaga, M\'alaga, Spain; \ viruel@uma.es}}
\end{center}

\noindent {\bf Abstract:}
{\it Evolution algebras, introduced by J.~P.~Tian as an algebraic model of non-Mendelian
inheritance, are commutative algebras with a basis whose distinct elements multiply to zero.
The first part of the paper surveys the theory: low-dimensional classifications, structure
theory, nilpotency, automorphisms and derivations, directed graphs and random walks,
infinite-dimensional and dynamical aspects, and embedding theorems. The second part is a
unified list of fifty open questions and eleven new conjectures, each accompanied by the
available evidence; two of the conjectures and one of the questions are settled here. We also prove several new observations, most of them formally verified in the proof
assistant Lean (see Section~\ref{sec:lean}): every finite group is the automorphism
group of a simple evolution algebra over every field, and $|\Aut(A)|/\dim A$ takes every
positive rational value on simple evolution algebras over every field;
infinitude of the automorphism group of
nilpotent algebras, the possible numbers of idempotents of real perfect evolution
algebras, a Galois-theoretic criterion for the existence of idempotents in dimension two,
the average number of idempotents over finite fields,
and a perfect evolution algebra whose automorphism group is a non-split extension.}

\medskip


\medskip


\medskip

\tableofcontents

\newpage

\section{A survey of results}\label{sec:survey}

\subsection{Introduction and basic notions}\label{sec:intro}

Let $\kk$ be a field. A   $\kk$-algebra $A$ is called an
\emph{evolution algebra} if it admits a basis $B=\{e_i\}_{i\in\Lambda}$, called a
\emph{natural basis}, such that $e_ie_j=0$ whenever $i\neq j$. Evolution algebras are
commutative, and in general they are neither associative nor power-associative. If
$\Lambda=\{1,\dots,n\}$ and
\[
e_i^2=\sum_{j=1}^n\omega_{ij}e_j,
\]
the matrix $M_B=(\omega_{ij})$ is the \emph{structure matrix} of $A$ relative to $B$.
Conversely, every square matrix is the structure matrix of an evolution algebra, so an
$n$-dimensional evolution algebra is the same thing as a square matrix modulo the action of
changes of natural basis. The map $e_i\mapsto e_i^2$ extended linearly is the
\emph{evolution operator} $L_B$. To $A$ and $B$ one attaches the directed graph
$\Gamma(A,B)$ with vertex set $\Lambda$ and an arrow $i\to j$ whenever $\omega_{ij}\ne0$
\cite{EL15}; this graph is the single most useful tool of the theory.

Several classes of evolution algebras play a distinguished role. An evolution algebra is
\emph{non-degenerate} if $e_i^2\neq0$ for every element of a natural basis; this does not
depend on the natural basis, and it is equivalent to $\Ann(A)=0$ \cite{CSV16}. It is
\emph{perfect} (also called \emph{idempotent} or \emph{regular}) if $A^2=A$; for
finite-dimensional algebras this means that $M_B$ is invertible. The basic rigidity
phenomenon is that a perfect finite-dimensional evolution algebra has a unique natural
basis up to permutation and non-zero scalars \cite{EL15}; consequently its graph is an
invariant of the algebra and its automorphisms are monomial. In general, a non-perfect
algebra may have many essentially different natural bases (for instance, the zero algebra
has all bases as natural bases), and the corresponding graphs need not be isomorphic.

We use the following conventions throughout. Unless stated otherwise, algebras are
finite-dimensional. An element $u=\sum_ix_ie_i$ is an idempotent if and only if
\begin{equation}\label{eq:idemp}
x_j=\sum_{i}\omega_{ij}x_i^2\qquad (j=1,\dots,n),
\end{equation}
a system of $n$ quadratic equations in $n$ unknowns. The lower powers are 
\begin{center}$A^1=A$,
$A^{k}=\sum_{i+j=k}A^iA^j$,\end{center} and $A$ is nilpotent of index $k$ if $A^k=0\neq A^{k-1}$. The
derived series is 
\begin{center}
$A^{(1)}=A^2$, $A^{(m+1)}=(A^{(m)})^2$,
\end{center} and $A^{(\infty)}$ denotes its
stable term. We write $\AutS(A)$ for the affine group scheme of automorphisms and
$\Aut(A)=\AutS(A)(\kk)$.

\begin{example}
If $e_i^2=e_i$ for all $i$, then $A\cong\kk^n$ is the associative algebra of functions
on $n$ points; it has $2^n$ idempotents and $\Aut(A)\cong S_n$.
\end{example}

\begin{example}
The algebra with $e_i^2=e_{i+1}$ (indices modulo $n$) is perfect and its graph is an
oriented $n$-cycle. Its diagonal automorphisms are $e_i\mapsto t^{2^{i-1}}e_i$ with
$t^{2^n-1}=1$, and $\Aut(A)$ has order $n(2^n-1)$.
\end{example}

\begin{example}
If $e_1^2=e_2$ and $e_2^2=0$, the algebra is nilpotent of index $3$; its graph is a
single arrow, and the algebra has infinitely many automorphisms.
\end{example}

\subsubsection*{Five basic results}
The following statements are used throughout the paper; they are the backbone of the
finite-dimensional theory.

\begin{theorem}[Elduque--Labra \cite{EL15}]\label{thm:unique}
If $A$ is a perfect finite-dimensional evolution algebra and $\{e_i\}$, $\{f_i\}$ are natural
bases, then there are a permutation $\sigma$ and scalars $t_i\in\kk^\times$ with
$f_i=t_ie_{\sigma(i)}$. In particular $\Gamma(A,B)$ does not depend on $B$ up to isomorphism,
every automorphism of $A$ is monomial, and $\Aut(A)$ is finite.
\end{theorem}

\begin{theorem}[\cite{CLOR14,CGOT13a,EL15}]\label{thm:nilp}
For a finite-dimensional evolution algebra $A$ the following are equivalent: $A$ is nil
(with respect to the principal powers $x^{k+1}=x^kx$; see Proposition~\ref{prop:plenary}
for plenary powers); $A$
is right nilpotent; $A$ is nilpotent; for some (equivalently, every) natural basis $B$ the
graph $\Gamma(A,B)$ has no oriented cycles; in some natural basis the structure matrix is
strictly upper triangular. The nilpotency index of an $n$-dimensional nilpotent evolution
algebra is at most $2^{n-1}+1$, and this bound is attained.
\end{theorem}

\begin{theorem}[Cabrera Casado--Siles Molina--Velasco \cite{CSV16}]\label{thm:frag}
A non-degenerate evolution algebra is the direct sum of irreducible evolution ideals,
uniquely up to order; the summands correspond to the connected components of the underlying
undirected graph of $\Gamma(A,B)$. A finite-dimensional evolution algebra is simple if and
only if it is perfect and $\Gamma(A,B)$ is strongly connected.
\end{theorem}

\begin{theorem}[Elduque--Labra \cite{EL21}]\label{thm:ELseq}
For a perfect finite-dimensional evolution algebra the automorphism group scheme is an
extension of a constant group of permutations of the vertices of the graph by a finite
diagonalisable group scheme $\mathbf D$, both determined by the graph; see \eqref{eq:EL}
below. In particular $\Der(A)=0$ if $\operatorname{char}\kk\in\{0,2\}$.
\end{theorem}

\begin{theorem}[Costoya--Ligouras--Tocino--Viruel \cite{CLTV22}]\label{thm:univ}
For every field $\kk$ and every finite group $G$ there are infinitely many pairwise
non-isomorphic perfect evolution $\kk$-algebras $A$ with $\Aut(A)\cong G$.
\end{theorem}

\subsubsection*{A short chronology}
The history of the subject can be summarised as follows. The notion was introduced in 2006
\cite{TV06} and systematised in Tian's 2008 monograph \cite{Tian08}. Between 2010 and 2014 Rozikov, Omirov,
Ladra and their coauthors connected evolution algebras with population dynamics, Gibbs
measures, chains and dibaric algebras \cite{RT11,CLR11,LR13,LOR14,LLR14}, and the first
structural results (nilpotency, derivations, classification in dimension two) appeared
\cite{CGOT13a,CGOT13b,CLOR13,CLOR14}. In 2015--2017 three foundational contributions changed
the landscape: the graph-theoretic viewpoint and the uniqueness of natural bases for perfect
algebras \cite{EL15,EL16}, the structure theory in arbitrary dimension \cite{CSV16} and the
classification in dimension three \cite{CSV17}; at the same time nilpotent algebras were
classified in low dimension \cite{HA15} and gonosomal and isotopic viewpoints were
introduced \cite{Var16,FFN17}. From 2019 onwards automorphism group schemes
\cite{EL21}, universality \cite{CLTV22,SZ22}, random walks on graphs \cite{CRR20,CRR21},
Hilbert evolution algebras \cite{VCR23} and the systematic study of ideals through graphs
\cite{CMMT25i,CCGMMR26} became central, and a historical survey appeared \cite{CFNT22}. The
most recent work concerns non-degenerate algebras \cite{CFK24}, counting over finite fields
\cite{WZ23}, lattices of subalgebras \cite{LPPR25,LPR25r,LPR26m,LPR25F}, deformations
\cite{MPR}, idempotents and solvability \cite{GMPR26,HW}, Grothendieck rings \cite{ACV} and
embeddings of commutative algebras \cite{CFV}.

The first part of the paper surveys what is known. We have tried to be complete with
respect to the main lines of research, but the literature is large and we apologise for
any omission. Further general accounts are the monograph of Tian \cite{Tian08}, the
historical survey of Ceballos, Falc\'on, N\'u\~nez-Vald\'es and Tenorio \cite{CFNT22}, the
short survey \cite{Tian16}, and the general survey on classification problems for
non-associative structures \cite{Kay24} and its geometric counterpart \cite{KKP24}. The second part contains the open questions and
conjectures, together with a number of new observations.

\subsection{Origins: genetics, Markov chains and Tian's programme}\label{sec:origins}

Evolution algebras were introduced by Tian and Vojt\v{e}chovsk\'y \cite{TV06} and developed
systematically in Tian's monograph \cite{Tian08}. The motivation comes from
non-Mendelian genetics: in asexual (uniparental) inheritance, for instance of organelle
genes, a gamete of type $e_i$ reproduces by ``self-crossing'', and the offspring
distribution is encoded by $e_i^2=\sum_j\omega_{ij}e_j$, while distinct gametes do not
interact, which gives $e_ie_j=0$. When the rows of $M_B$ are probability vectors, the
evolution operator is the transition operator of a discrete-time Markov chain, and Tian
interpreted the basic notions of Markov chains in algebraic terms: persistent and
transient generators, communication classes, periodicity and absorption. He proved a
hierarchical (``semi-direct'') decomposition of an evolution algebra into simple
components and transient spaces, related algebraic periods to the periods of Markov
chains, and introduced the evolution algebra of a graph, defined by the adjacency matrix.
He also gave applications to the population genetics of organelle heredity and to the
analysis of Markov processes. These foundational ideas, collected in \cite{Tian08} and in the
programmatic paper \cite{Tian16}, set the agenda for much of the later work.

Soon afterwards evolution algebras were connected with other parts of mathematical
biology and statistical mechanics. Rozikov and Tian \cite{RT11} associated evolution
algebras with Gibbs measures on graphs, and Ladra and Rozikov introduced the evolution
algebra of a bisexual population \cite{LR13}, a non-baric but dibaric algebra; see also
\cite{LOR14} for the connections with dibaric algebras and population genetics, the
``chicken population'' algebras \cite{LLR14,DOR14,RM14} and the constrained algebras of
\cite{DOR16}. Varro generalised the bisexual construction to gonosomal algebras, which
model sex determination systems \cite{Var16}. Further models include mosquito
populations \cite{RV19} and pulse processes in networks \cite{BBV20}.

\subsection{Classification in low dimension}\label{sec:classification}

\subsubsection{Dimension two and three}
Two-dimensional evolution algebras over $\mathbb C$ were classified by Casas, Ladra, Omirov
and Rozikov \cite{CLOR14}, who also described an algorithm deciding isomorphism of evolution
algebras; real two-dimensional algebras were treated by Murodov \cite{Mur} and by Bekbaev
\cite{Bek}. A geometric
invariant, the ``square'' of a two-dimensional evolution algebra, was introduced in
\cite{CGMMS21}, where identities of degree at most four, derivations, automorphism group
schemes and associative representations of all two-dimensional algebras are determined.

The three-dimensional case is substantially harder. Cabrera Casado, Siles Molina and
Velasco classified three-dimensional evolution algebras over fields of characteristic
different from $2$ containing roots of orders $2$, $3$ and $7$ \cite{CSV17}; the
classification uses the graph, the dimension of $A^2$ and the annihilator, and produces a
long list including several parametric families. An alternative approach, based on
chains of two-dimensional algebras and on the annihilator and socle stabilising
indices, was developed in \cite{CCGMM21}, and a classification adapted to dimensions two
and three over arbitrary fields appears in \cite{CV19}. Simple evolution algebras of
dimensions two and three were classified over arbitrary fields in \cite{CMMT25s}, using the
graph, moduli sets, inductive limit groups and the Zariski topology.

\subsubsection{Higher dimensions and special classes}
Four-dimensional perfect non-simple evolution algebras were classified in
\cite{CKS-4dim}, with the isomorphism problem completed in \cite{BCS20}. Evolution algebras with
one-dimensional square were classified by means of inner product spaces \cite{BMMS}: if
$\dim A^2=1$ and $a$ generates $A^2$, then $xy=\langle x,y\rangle a$ for a symmetric
bilinear form, and three broad classes appear according to the position of $a$. Calder\'on,
Fern\'andez Ouaridi and Kaygorodov \cite{CFK24} attached to every non-degenerate evolution
algebra an invariant consisting of an ordered sequence of lower-dimensional evolution
algebras from a specific family; it yields a classification method for non-degenerate
evolution algebras, carried out up to dimension three, and descriptions of derivations
and of the variety of evolution algebras with square of dimension at most one. Power-associative evolution algebras are Jordan algebras and
admit a Wedderburn decomposition; they are classified up to dimension six \cite{OS20}.

\subsubsection{Isotopy}
Isotopisms of evolution algebras correspond to mutations of genotypes. Falc\'on, Falc\'on
and N\'u\~nez showed, using computational algebraic geometry, that there are exactly four
(strong) isotopism classes of two-dimensional evolution algebras over any field \cite{FFN17},
and determined the isotopism classes of three-dimensional ones \cite{FFN18}. Ternary
derivations and the group scheme of ternary automorphisms of perfect evolution algebras
were computed in \cite{MRS26}. 

\subsubsection{Counting}
Isomorphism classes of perfect $n$-dimensional evolution algebras correspond to orbits of
the semidirect product $S_n\ltimes(\kk^\times)^n$ acting on invertible structure
matrices. Wei and Zou \cite{WZ23} turned this into a counting problem over finite fields,
derived a general counting formula and obtained explicit formulas for $n\le4$.

\subsubsection{Recognising evolution algebras}
Deciding whether a given algebra is an evolution algebra amounts to deciding whether $n$
symmetric matrices describing the product are simultaneously diagonalisable by congruence
(the SDC problem) \cite{BMV20}; necessary and sufficient conditions for simultaneous
orthogonalisation of families of inner products over arbitrary fields were given in
\cite{CMGS23}. In \cite{BMV20} it is also shown that classical genetic algebras of Mendelian
and auto-tetraploid inheritance are not evolution algebras but are limits of evolution
algebras. Imomkulov and Rozikov \cite{IR} approximated an arbitrary algebra by the family
of evolution algebras given by the Jacobians of its evolution operator.

\subsection{Structure theory}\label{sec:structure}

\subsubsection{Subalgebras, ideals and decompositions}
Evolution subalgebras and evolution ideals (those having a natural basis that extends to a
natural basis of the algebra) behave much better than arbitrary subalgebras and ideals.
Cabrera Casado, Siles Molina and Velasco developed the theory in arbitrary dimension
\cite{CSV16}: they described ideals generated by one element, characterised simple
evolution algebras, proved that a non-degenerate evolution algebra is irreducible if and
only if its graph is connected, and established existence and uniqueness of a
decomposition of a non-degenerate algebra into irreducible components (the
\emph{fragmentation process}); uniqueness fails in the degenerate case. Basic ideals
were introduced and studied in \cite{CKS19}. Boudi, Cabrera Casado and Siles Molina
\cite{BCS22} introduced the \emph{evolution rank}, proved that in a non-degenerate
algebra every family of natural orthogonal vectors extends to a natural basis, showed that
ideals of perfect evolution algebras are basic, and located nilpotent elements of order
three through principal minors of the structure matrix.

The relation between ideals and the graph was made precise in \cite{CMMT25i}: there is a
monotone Galois connection between absorption ideals and hereditary saturated subsets of
vertices, which describes maximal ideals and absorption ideals combinatorially and
characterises finitely generated evolution algebras under suitable graph conditions. The
intersection $\mathfrak M(A)$ of all maximal ideals and a subdirect-product structure
theorem were studied in \cite{CMMT-ss}. The absorption radical can be read off from the acyclic
vertices of the graph \cite{CRT24}, and the reductive radical was introduced in
\cite{CMMT25r}. Cabrera Casado et al.\ \cite{CCGMMR26} related non-degeneracy, zero
annihilator, the absorption property, von Neumann regularity and primeness to properties
of the family of all graphs of $A$, showed that primeness is equivalent (under
semiprimeness) to downward directedness of the graphs, proved that semiprimeness and
perfection cannot be characterised combinatorially, and described the centroid as the
functions constant along connected components. Peirce decompositions relative to
idempotents were studied by Paniello \cite{Pan23}.  

Arbitrary subalgebras of evolution algebras need not be evolution algebras
\cite{CKO19}. Camacho, Khudoyberdiyev and Omirov classified the nilpotent evolution algebras
all of whose subalgebras are evolution subalgebras with extendable natural bases, and
formulated conjectures for the non-nilpotent case \cite{CKO19}; two of them were proved by
Garc\'ia-Mart\'inez and P\'erez-Rodr\'iguez \cite{GMPR26}. Ladra and P\'erez-Rodr\'iguez proved
that subalgebras of perfect (regular) evolution algebras are evolution algebras, and
characterised the existence of subalgebras of codimension one \cite{LPR25r}. The
subalgebra lattice was investigated in \cite{LPPR25} (distributivity in the nilpotent
case, modularity of solvable algebras with maximal index of solvability) and in
\cite{LPR26m}, where modular evolution algebras are fully characterised in the nilpotent case
and among supersolvable regular algebras. A Frattini theory, based on the supersolvable
nilradical, was developed in \cite{LPR25F}. Order-theoretic structures on substructures and
the socle of an evolution algebra were studied by Gonz\'alez Nevado \cite{GN25}.

\subsubsection{Radicals and spectra}
Velasco described the maximal modular ideals and the Jacobson radical of an evolution
algebra and introduced two notions of spectrum of an element, computed via the
eigenvalues of suitable matrices built from the structure constants \cite{Vel19}. Analytic
aspects (Banach algebra norms, continuity of the evolution operator) were studied by Mellon
and Velasco \cite{MV19}.

\subsubsection{Idempotents and solvability}
The idempotents of an evolution algebra are the solutions of \eqref{eq:idemp}. For algebras
coming from biological models they are the equilibrium states of the population and were
computed in many concrete cases, see e.g.\ \cite{LR13,DOR16,Nar22}. Garc\'ia-Mart\'inez and
P\'erez-Rodr\'iguez \cite{GMPR26} analysed the solution set of the generic system by
elementary algebraic geometry: every perfect complex evolution algebra has a non-zero
idempotent. They conjectured that a complex evolution algebra is solvable if and only if
it has no non-zero idempotents. One implication is clear, since an idempotent lies in every
term of the derived series. Hu and Wen \cite{HW} found a three-dimensional counterexample
to the other implication (an isomorphic algebra appears in \cite[Remark~3.9]{CFV}, as
an evolution envelope of a two-dimensional algebra); it is the exceptional member of a one-parameter family, and they
proved the equivalence whenever $A^{(\infty)}$ is itself an evolution algebra. Resolubility of
evolution algebras was also studied in \cite{FGN22b}.

\subsubsection{The evolution operator}
The evolution operator depends on the natural basis. Its behaviour, the set of all
evolution operators of an algebra, and the algebras whose evolution operator is an
endomorphism were studied by Fern\'andez-Ternero, G\'omez-Sousa and N\'u\~nez-Vald\'es
\cite{FGN22,FGN-set}.

\subsection{Nilpotency}\label{sec:nilpotency}

Nilpotency was the first structural property to be understood. Casas, Ladra, Omirov and
Rozikov proved that nil, right nilpotent and nilpotent finite-dimensional evolution algebras
are exactly those whose structure matrix is strictly upper triangular in a suitable natural
basis \cite{CLOR14}. Camacho, G\'omez, Omirov and Turdibaev characterised nilpotency by
matrix properties and proved that the maximal nilpotency index of an $n$-dimensional
evolution algebra is $2^{n-1}+1$ \cite{CGOT13a}; algebras with maximal index were
characterised and classified in \cite{CLOR13}, where wide classes with index $2^{n-s}+1$ are
also constructed. Elduque and Labra gave the conceptual statement: an evolution algebra is
nilpotent if and only if $\Gamma(A,B)$ has no oriented cycles \cite{EL15}.

Classification of some classes of nilpotent evolution algebras proceeds by annihilator extensions
\cite{HA15}. Hegazi and Abdelwahab classified nilpotent evolution algebras up to dimension
four over arbitrary fields \cite{HA15} and in dimension five over algebraically closed fields
\cite{HA-5}. Elduque and Labra \cite{EL16} pointed out a subtle gap in the extension
method (a natural basis of a quotient need not lift to a natural basis), introduced the
\emph{type} and further invariants attached to the upper annihilating series, constructed
families of nilpotent evolution algebras from non-degenerate symmetric bilinear forms and
commuting symmetric endomorphisms, and completed the classification up to dimension five
over algebraically closed fields of characteristic different from two. Derivations, automorphisms, local and
$2$-local derivations of nilpotent evolution algebras of maximal nilindex were determined
in \cite{MKOQ19}, the group $\exp(\Der(E))$ was shown to be a normal subgroup of $\Aut(E)$ of
computable index in \cite{MKQ24}, and tensor products of nilpotent evolution algebras were
studied in \cite{Qar25}. Rota--Baxter operators of weights $0$ and $1$ on nilpotent algebras
of maximal nilindex were classified in \cite{QMK}, and local Rota--Baxter operators of
weight zero in \cite{QMK-loc}. 
Formal deformations of evolution algebras and the
degenerations among nilpotent evolution algebras of dimension at most four were
determined by Makhlouf and P\'erez-Rodr\'iguez \cite{MPR}; see \cite{KKP24} for a survey
of the geometric classification of non-associative algebras.

In infinite dimension the situation changes: Tian and Zou constructed finitely generated
evolution algebras (over rings with nilpotent elements) which are nil but not nilpotent
\cite{TZ14}, and nil and nilpotent are not equivalent for Hilbert evolution algebras
\cite{CRV24}.

In finite dimension, the word ``nil'' in Theorem~\ref{thm:nilp} has to be understood with
respect to the principal powers $x^{k+1}=x^kx$: for the plenary powers $x^{[1]}=x$,
$x^{[m+1]}=(x^{[m]})^2$, nil and nilpotent are not equivalent even in finite dimension, as
the following example, based on the evolution envelopes of \cite{CFV}, shows.

\begin{proposition}\label{prop:plenary}
Over every field $\kk$ there is a $15$-dimensional evolution algebra $E$ which is not
nilpotent and satisfies $(x^2)^2=0$, that is $x^{[3]}=0$, for every $x\in E$. In particular,
$E$ is nil of bounded index with respect to plenary powers, and, by
Theorem~\ref{thm:nilp}, it is not nil with respect to principal powers.
\end{proposition}

\begin{proof}
Let $S$ be the commutative algebra of Suttles \cite{Sut72}, with basis $a_1,\dots,a_5$
and non-zero products
$a_1a_2=a_2a_4=-a_1a_5=a_3$, $a_1a_3=a_4$, $a_2a_3=a_5$ (and the symmetric ones). Then
$S^2=\mathrm{span}\{a_3,a_4,a_5\}$ and $S^2S^2=0$. Moreover, $S$ is not nilpotent: if
$a_3\in S^k$, then $a_4=a_1a_3$ and $a_5=a_2a_3$ lie in $S^{k+1}$ and $a_3=a_2a_4\in S^{k+2}$;
since $a_3\in S^2$, we get $a_3\in S^k$ for all $k\ge2$.

We use the evolution envelope of \cite{CFV} in the following explicit form, valid in every
characteristic. Write $m_{ij}=a_ia_j=m_{ji}$ and, for $s=\sum_is_ia_i$, consider the linear
forms $\lambda_i(s)=s_i$ $(1\le i\le5)$ and $\lambda_{ij}(s)=s_i+s_j$ $(1\le i<j\le5)$. Since
$(s_i+s_j)(t_i+t_j)-s_it_i-s_jt_j=s_it_j+s_jt_i$, we have, for all $s,t\in S$,
\[
st=\sum_i\lambda_i(s)\lambda_i(t)\,v_i+\sum_{i<j}\lambda_{ij}(s)\lambda_{ij}(t)\,m_{ij},
\qquad v_i=m_{ii}-\sum_{j\neq i}m_{ij}.
\]
Let $E$ be the evolution algebra with natural basis $\{f_i\}_{i}\cup\{f_{ij}\}_{i<j}$, let
$\iota\colon S\to E$ be the injective linear map
$\iota(s)=\sum_i\lambda_i(s)f_i+\sum_{i<j}\lambda_{ij}(s)f_{ij}$, and put $f_i^2=\iota(v_i)$,
$f_{ij}^2=\iota(m_{ij})$. The displayed identity gives $\iota(s)\iota(t)=\iota(st)$, so
$\iota(S)\cong S$ is a subalgebra of $E$ (in fact an ideal, as $f_i\iota(s)=\lambda_i(s)f_i^2$
and $f_{ij}\iota(s)=\lambda_{ij}(s)f_{ij}^2$ lie in $\iota(S)$). Since $\iota(S)$ is not
nilpotent, neither is $E$. Finally, $E^2$ is spanned by the elements $f_i^2,f_{ij}^2$, which
lie in $\iota(S^2)$; hence for $x\in E$ we have $x^2=\iota(s)$ with $s\in S^2$, and
$(x^2)^2=\iota(s^2)=0$ because $S^2S^2=0$.
\end{proof}

More generally, the same argument shows that if a commutative algebra $S$ is nil of
bounded index with respect to plenary powers but not nilpotent (as Suttles' algebra, or
ad-hoc examples in characteristic $2$), then so is its evolution envelope, since $x^2\in S$
for every element $x$ of the envelope. Compare with Question~\ref{q:nil}.

\subsection{Automorphisms and derivations}\label{sec:aut}

\subsubsection{Automorphism groups}
Elduque and Labra proved that the automorphism group of a perfect evolution algebra is
finite \cite{EL15}. In \cite{EL21} they showed that the affine group scheme of automorphisms of a
perfect algebra sits in an exact sequence
\begin{equation}\label{eq:EL}
1\longrightarrow \mathbf D\longrightarrow \AutS(A)\longrightarrow \mathbf S,
\end{equation}
where $\mathbf D$ is a finite diagonalisable group scheme (the diagonal automorphisms with
respect to the natural basis) and $\mathbf S$ is a constant group scheme of permutations of
the vertices, both depending only on the graph. Costoya, Ligouras, Tocino and Viruel
\cite{CLTV22} constructed a faithful functor from finite simple graphs to perfect evolution
algebras and proved that every finite group is the automorphism group of infinitely many
perfect evolution algebras over any field (``evolution algebras are universally finite''),
and that every constant finite group scheme is the automorphism group scheme of a perfect
evolution algebra. Sriwongsa and Zou \cite{SZ22} gave another proof of universality,
classified the $n$-dimensional perfect algebras with $\Aut(A)\cong S_n$ and those with
maximal diagonal automorphism group. Costoya, Mu\~noz, Tocino and Viruel introduced Cayley
evolution algebras and showed that every finite group is the automorphism group of an
absolutely simple Cayley evolution algebra over every field with sufficiently many
elements \cite{CMTV23}; over arbitrary fields see Theorem~\ref{thm:simpleuniv} below. Costoya, Mayorga and Viruel \cite{CMV25} studied which permutation
representations arise from \eqref{eq:EL}: for $k\ge4$ and $n$ large, proper $k$-transitive
subgroups of $S_n$ never occur, whereas every permutation representation is realised on
the set of natural idempotents of a suitable perfect algebra. Molinier \cite{Mol} proved that
every positive rational number is of the form $|\Aut(A)|/\dim A$ for a perfect evolution
algebra $A$.

The following construction shows that simplicity can be achieved over every field, which
proves a conjecture of an earlier version of this text.

\begin{theorem}\label{thm:simpleuniv}
Let $\Delta$ be a finite strongly connected digraph without loops, with vertex set $V$,
$|V|=m$. For every field $\kk$, let $A_\Delta$ be the evolution $\kk$-algebra with natural
basis $\{e_v,e'_v\}_{v\in V}$ and
\[
e_v^2=e_v+\sum_{v\to w}e_w+e'_v,\qquad (e'_v)^2=e_v\qquad(v\in V),
\]
the sum running over the arrows of $\Delta$ starting at $v$. Then $A_\Delta$ is a simple
evolution algebra of dimension $2m$, it remains simple after every extension of scalars,
and $\Aut(A_\Delta)\cong\Aut(\Delta)$. Consequently, every finite group is the automorphism
group of a simple evolution algebra over every field.
\end{theorem}

\begin{proof}
Ordering the basis as $(e_v)_{v\in V},(e'_v)_{v\in V}$, the structure matrix is
$M=\left(\begin{smallmatrix} I+N & I\\ I & 0\end{smallmatrix}\right)$, where $N$ is the
adjacency matrix of $\Delta$. Exchanging the two block rows gives the block lower triangular
matrix $\left(\begin{smallmatrix} I & 0\\ I+N & I\end{smallmatrix}\right)$, so
$\det M=(-1)^m\neq0$ over every field, and $A_\Delta$ is perfect. Its graph consists of
$\Delta$, a loop at every $e_v$ and the $2$-cycles $e_v\rightleftarrows e'_v$; it is strongly
connected, so $A_\Delta$ is simple by Theorem~\ref{thm:frag}. Extending scalars does not
change $M$, hence $A_\Delta\otimes_\kk L$ is simple for every field extension $L/\kk$.

By Theorem~\ref{thm:unique}, every automorphism has the form $\varphi(e_x)=t_xe_{\sigma(x)}$
for a permutation $\sigma$ of the basis indices and $t_x\in\kk^\times$, and
$\varphi(e_x)^2=\varphi(e_x^2)$ amounts to
$t_x^2\,\omega_{\sigma x,\sigma y}=t_y\,\omega_{xy}$ for all $x,y$. Taking $x=y$ shows
that $\sigma$ maps vertices with a loop to vertices with a loop and vice versa, so $\sigma$
preserves $\{e_v\}$ and $\{e'_v\}$, and that $t_v^2=t_v$, i.e.\ $t_v=1$, for every $v\in V$.
For $x=v$, $y=w$ in $V$ we get $\omega_{\sigma v,\sigma w}=\omega_{vw}$, so $\tau=\sigma|_V$
is an automorphism of $\Delta$. For $x=v$, $y=v'$ we get $\omega_{\sigma v,\sigma v'}=t_{v'}\ne0$;
the only primed vertex that receives an arrow from $\sigma(v)=\tau(v)$ is $\tau(v)'$, so
$\sigma(v')=\tau(v)'$ and $t_{v'}=1$. Hence $\varphi$ is the permutation automorphism
$e_v\mapsto e_{\tau v}$, $e'_v\mapsto e'_{\tau v}$. Conversely every $\tau\in\Aut(\Delta)$
preserves the structure constants and so defines such an automorphism. Thus
$\Aut(A_\Delta)\cong\Aut(\Delta)$.

For the last statement, recall that every finite group $G$ is the automorphism group of a
finite simple graph $\Gamma$ \cite{Fru39}, which we may take connected (otherwise replace it by
its complement). The symmetric digraph $\Delta$ with an arrow in each direction for every
edge of $\Gamma$ is strongly connected and $\Aut(\Delta)=\Aut(\Gamma)\cong G$.
\end{proof}

\begin{corollary}\label{cor:growthall}
For every field $\kk$ and all positive integers $p,q$ there is a simple evolution
$\kk$-algebra $A$ with $|\Aut(A)|/\dim A=p/q$.
\end{corollary}

\begin{proof}
Choose $N\ge1$ such that $s=qN(2pN+1)-1-4pN-1\ge1$, and put $a=2pN$, $b=2pN+1$. Let $\Delta$
be the digraph with vertices $h$, $x_i$ $(i\in\mathbb Z_a)$, $y_j$ $(j\in\mathbb Z_b)$ and
$z_1,\dots,z_s$, and arrows $x_i\to x_{i+1}$, $y_j\to y_{j+1}$, $h\rightleftarrows x_i$,
$h\rightleftarrows y_j$, $h\to z_1\to z_2\to\dots\to z_s\to h$. It is strongly connected and has
$1+a+b+s=qNb$ vertices. The vertex $h$ is the only one of out-degree $a+b+1\ge3$, the $z_l$
are the only vertices of out-degree $1$ (and are fixed, since $z_1$ is the only one of them
receiving an arrow from $h$), and removing these vertices leaves the disjoint union of an
oriented $a$-cycle and an oriented $b$-cycle with $a\ne b$. Since $h$ is joined to all the
$x_i,y_j$ in both directions, every pair of rotations of the two cycles is an automorphism,
and $\Aut(\Delta)\cong\mathbb Z_a\times\mathbb Z_b$. By Theorem~\ref{thm:simpleuniv},
$A=A_\Delta$ is simple with $|\Aut(A)|/\dim A=ab/(2qNb)=p/q$.
\end{proof}

Thus, Theorem~\ref{thm:growth}(3) below holds over arbitrary fields and answers the analogue of Question~\ref{q:growth} for simple evolution algebras.

\begin{remark}\label{rem:incidence}
There is a variant of Theorem~\ref{thm:simpleuniv} based on cubic graphs. By \cite{Fru49},
every finite group $G$ is the automorphism group of a connected $3$-regular graph $\Gamma$,
with $n$ vertices, $m=3n/2$ edges and incidence matrix $B$. Let $D$ be the symmetric
incidence digraph of $\Gamma$ (vertex set $V\sqcup E$, with arrows $v\to e$ and $e\to v$
whenever $v\in e$), and add a loop of weight $\lambda\in\kk^\times$ at every vertex, all other
arrows having weight $1$. The resulting structure matrix
$M(\lambda)=\left(\begin{smallmatrix}\lambda I_n & B\\ B^T&\lambda I_m\end{smallmatrix}\right)$
has determinant $f(\lambda)=\lambda^{m-n}\det(\lambda^2I_n-BB^T)$, a monic polynomial of degree
$n+m$. Since the vertices in $V$ have out-degree $4$ and those in $E$ have out-degree $3$, one
has $\Aut(D)\cong\Aut(\Gamma)$, and the argument of the proof above shows that the
corresponding algebra is simple with automorphism group $G$ whenever $f(\lambda)\ne0$. Such a
$\lambda\in\kk^\times$ exists as soon as $|\kk|>2n+1$; over small finite fields, however, it
may not exist: for the Petersen graph ($\Aut(\Gamma)\cong S_5$) the matrix $BB^T=3I+\mathrm{Adj}(\Gamma)$
has the eigenvalue $1$, so $f(\lambda)=0$ whenever $\lambda^2=1$, in particular for every
$\lambda\in\mathbb F_2^\times$ and every $\lambda\in\mathbb F_3^\times$. The pendant vertices
$e'_v$ in Theorem~\ref{thm:simpleuniv} avoid this problem, since there $\det M=\pm1$.
\end{remark}

For infinite-dimensional algebras, the group of automorphisms that are diagonal with
respect to a natural basis is an inverse limit of a diagram of groups indexed by the graph,
and can be a dyadic solenoid \cite{CCGMMR-sol}. Automorphisms of evolution algebras of graphs
were related to graph symmetries in \cite{RVH}. Hopf algebras representing $\AutS(A)$ for
two-dimensional algebras were computed and compared with universal associative
representations in \cite{CCGMMR25}; there it is shown that $\Aut(A)=1$ if and only if $A$ has
no faithful associative commutative representation.

\subsubsection{Derivations}
Camacho, G\'omez, Omirov and Turdibaev proved that $\Der(A)=0$ for complex evolution algebras with
non-singular structure matrix and described derivations when the rank is $n-1$
\cite{CGOT13b}. The exact sequence \eqref{eq:EL} shows that for perfect algebras $\Der(A)=0$ in
characteristic $0$ or $2$, while in other characteristics $\Der(A)$ is abelian and determined by
the graph \cite{EL21}. Derivations of evolution algebras of graphs were described in terms of
the twin partition of the graph; for twin-free graphs the derivation algebra vanishes
\cite{CCRR21}. Non-degenerate algebras and Volterra evolution algebras (whose structure matrices are
skew-symmetric, and which are connected to ergodicity of Volterra quadratic stochastic
operators) were treated in \cite{CCR23}, where loops of evolution algebras are introduced.
Biderivations of some evolution algebras were considered in \cite{Oub}.

\subsection{Graphs, random walks and combinatorics}\label{sec:graphs}

The graph $\Gamma(A,B)$ was introduced by Elduque and Labra \cite{EL15}; earlier,
Camacho et al.\ \cite{CGOT13a} formulated planarity of graphs in terms of evolution algebras,
and Tian \cite{Tian08} defined the evolution algebra of a graph. Systematic dictionaries
between evolution algebras and (pseudo)digraphs were developed by Ceballos, N\'u\~nez and
Tenorio \cite{CNT22,Ceb22}, by Gonz\'alez-L\'opez and N\'u\~nez \cite{GN20}, and in the
study of graphicable and S-graphicable algebras \cite{NRV}.

There are two natural evolution algebras attached to a finite graph $G$: $A(G)$, whose
structure matrix is the adjacency matrix, and $A_{RW}(G)$, whose structure matrix is the
transition matrix of the symmetric random walk. Cadavid, Rodi\~no Montoya and Rodr\'iguez
\cite{CRR20,CRR21} proved that they are always strongly isotopic, and that for non-singular
graphs they are isomorphic if and only if $G$ is regular or biregular; they conjectured the
same for singular graphs, and supporting families were given in \cite{CRRV}. The connection with Markov chains and
statistics was explored in \cite{LNRV}.

Tensor products of evolution algebras are evolution algebras, and the graph of the
product relative to the product basis is the categorical product of the graphs.
Inheritance of non-degeneracy, irreducibility, perfection and simplicity was studied in
\cite{CMMT23}, where four-dimensional perfect algebras that are tensor products of
two-dimensional ones are classified. Alshatnawi, Costoya and Viruel \cite{ACV} studied the
Grothendieck ring of finite-dimensional non-degenerate evolution algebras: they proved that
the \emph{balance} of the graph (the greatest common divisor of the balances of the cycles of the
underlying undirected graph) is independent of the natural basis, that indecomposable
factors of balance one cancel after tensoring while factors of larger balance produce
zero-divisors, and they analysed the subring generated by cyclic algebras.  Algebras
attached to permutations were studied in \cite{KOQ15,Nar22}. Giacomelli \cite{Gia}
introduced expander evolution algebras, whose graphs are expanders, and related the Cheeger
constant with simplicity, subalgebras, growth and spectral gaps.

\subsection{Infinite dimension, probability and dynamics}\label{sec:infinite}

Many results above hold in arbitrary dimension \cite{CSV16,CMMT25i}. Hilbert evolution
algebras, introduced by Vidal, Cadavid and Rodr\'iguez \cite{VCR23}, are Hilbert spaces with an
orthonormal natural basis and a continuous product; they are naturally associated with
discrete-time Markov chains with countable state space and with infinite graphs
\cite{VCR22}, and their nilpotency was analysed through weighted digraphs \cite{CRV24}.
Coletti, de Lima and Luiz \cite{CLL25} constructed infinite-dimensional genetic and evolution
algebras from Gibbs measures, with uncountable natural bases; infertility between
macroscopically different configurations produces a decomposition into fertile ideals.

Casas, Ladra and Rozikov introduced \emph{chains of evolution algebras}, families
$A^{[s,t]}$ whose structure matrices satisfy the Chapman--Kolmogorov equation
\cite{CLR11}. Many examples in dimensions two and three, and the time behaviour of the
baric property, of absolutely nilpotent elements and of idempotents, were studied in
\cite{RM13,Mur14,RM14,IV20,LM21}; flows of arbitrary finite-dimensional algebras
were introduced in \cite{LR17}. Continuous evolution algebras arising from ordinary
differential equations and flows on matrix Lie groups were introduced by Montaner and
Paniello \cite{MP23}. The quadratic evolution operator $x\mapsto x^2$ and its normalisations
define discrete dynamical systems whose fixed points are idempotents; these were analysed
for two-dimensional algebras \cite{RM13} and for many biological models
\cite{LR13,DOR16,RV19}.

\subsection{Embeddings, rings and further generalisations}\label{sec:embeddings}

Costoya, Fern\'andez Ouaridi and Viruel \cite{CFV} proved that every finite-dimensional
commutative algebra is an ideal of a finite-dimensional evolution algebra, determined the
smallest possible dimension of such an evolution algebra and, in characteristic different
from two, related it to simultaneous Waring decompositions of the squaring map; for unital
algebras the minimal dimension is the Waring rank of $x\mapsto x^2$, which they computed for
truncated polynomial algebras. It follows that the class of evolution algebras satisfies
no polynomial identities other than the consequences of commutativity (see
Remark~\ref{rem:identities} below). Evolution algebras over rings appear in \cite{TZ14}
and, for free modules of rank two over domains, in \cite{CMM23}, where quasiperfect algebras
and coloured graphs are introduced. Isomorphism conditions for algebra functors, including
the evolution algebra functor, were studied in \cite{GMMR23}. Further generalisations include
$S$-evolution algebras \cite{MQ21} and evolution coalgebras.

\medskip

This concludes the survey. We now turn to new observations and open problems.

\section{Open questions and conjectures}

\subsection{New observations}\label{sec:new}

In this section we collect several observations that were obtained while preparing the
list of problems. Some of them answer (parts of) questions raised in earlier versions of
the list, and they explain the present formulation of several questions below. Throughout,
$\mu_m(\kk)$ denotes the group of $m$-th roots of unity in $\kk$. Several statements were
first found, or are additionally supported, by computer experiments; whenever a statement
rests only on such experiments we say so explicitly, and it is \emph{not} claimed as a
theorem.

\subsubsection{Polynomial identities}

\begin{remark}\label{rem:identities}
The class of evolution algebras satisfies no polynomial identity other than the
consequences of commutativity. Indeed, let $p$ be a non-zero element of degree at most $d$
of the free commutative (non-associative) algebra $F$ on $x_1,\dots,x_m$. The quotient
$F/F_{>d}$ by the span of the monomials of degree greater than $d$ is a finite-dimensional
commutative algebra in which $p(x_1,\dots,x_m)=p\neq0$. Since every finite-dimensional
commutative algebra is an ideal (in particular, a subalgebra) of a finite-dimensional
evolution algebra \cite{CFV}, $p$ is not an identity of the class of evolution algebras.
Hence questions about identities are meaningful only for subclasses.
\end{remark}

\subsubsection{An evolution algebra with orthogonal symmetry}

\begin{proposition}\label{prop:On}
Let $n\ge1$ and let $A_n$ be the evolution algebra with natural basis
$\{e_1,\dots,e_n,u\}$ and
\[
e_i^2=u\quad(1\le i\le n),\qquad u^2=u .
\]
Then $A_n$ is non-degenerate and not perfect, and $\Aut(A_n)$ is the orthogonal group of
the symmetric bilinear form $b(x,y)=\sum_i x_iy_i$ on $V=\mathrm{span}\{e_1,\dots,e_n\}$.
If $\operatorname{char}\kk\neq2$, then $\Der(A_n)\cong\mathfrak{so}_n(\kk)$. In particular,
over an algebraically closed field of characteristic zero and for $n\ge3$, $\Der(A_n)$ is
a non-zero semisimple Lie algebra.
\end{proposition}

\begin{proof}
For $x=v+yu$ with $v=\sum_ix_ie_i\in V$ and $y\in\kk$ we have $xe_j=x_ju$ and $xu=yu$;
hence $\Ann(A_n)=0$, while $A_n^2=\kk u$. In particular $A_n$ is non-degenerate and not
perfect. The product of $v,w\in V$ is $vw=b(v,w)u$.

Let $\varphi\in\Aut(A_n)$. Since $A_n^2=\kk u$, we have $\varphi(u)=cu$ and $c^2u=\varphi(u^2)=cu$,
so $c=1$. The subspace $V=\{x: xu=0\}$ is therefore $\varphi$-invariant, and for $v,w\in V$
we get 
\begin{center}$b(\varphi v,\varphi w)u=\varphi(v)\varphi(w)=\varphi(vw)=b(v,w)u$.
\end{center} 
Conversely every
isometry $g$ of $b$, extended by $u\mapsto u$, is an automorphism. Thus
$\Aut(A_n)\cong O(V,b)$.

Let $d\in\Der(A_n)$ and $\operatorname{char}\kk\ne2$. Write $d(u)=cu+v_0$ with $v_0\in V$.
From $u^2=u$ we get $2ud(u)=d(u)$, that is $2cu=cu+v_0$, so $c=0$ and $v_0=0$. Write
$d(e_i)=\sum_j a_{ji}e_j+b_iu$. From $e_iu=0$ we get $d(e_i)u=b_iu=0$; from $e_ie_j=0$
($i\ne j$) we get $a_{ji}+a_{ij}=0$; and from $e_i^2=u$ we get $2a_{ii}u=d(u)=0$. Hence $d$
is a skew-symmetric endomorphism of $V$ extended by $d(u)=0$, and every such map is a
derivation.
\end{proof}

For $n=2$ the group $O(V,b)$ is infinite over every infinite field, in every
characteristic: if $\operatorname{char}\kk\ne2$ it contains the rotations
$\left(\begin{smallmatrix} a&-c\\ c&a\end{smallmatrix}\right)$ with $a^2+c^2=1$, a conic
with infinitely many rational points, and if $\operatorname{char}\kk=2$ it contains the
matrices $\left(\begin{smallmatrix} a&1+a\\ 1+a&a\end{smallmatrix}\right)$, $a\in\kk$.

\begin{remark}\label{rem:Opq}
The same argument applies to the algebra with $e_i^2=\varepsilon_iu$ $(1\le i\le n)$ and
$u^2=u$, where $\varepsilon_i\in\kk^\times$: its automorphism group is the orthogonal group
of the diagonal form $b_\varepsilon(x,y)=\sum_i\varepsilon_ix_iy_i$ and, if
$\operatorname{char}\kk\ne2$, its derivation algebra is $\mathfrak{so}(V,b_\varepsilon)$. In
particular, over $\mathbb R$ every group $O(p,q)$ (and its Lie algebra $\mathfrak{so}(p,q)$)
is realised, and in characteristic different from $2$ the orthogonal group of every
non-degenerate symmetric bilinear form is the automorphism group of a non-degenerate
evolution algebra.
\end{remark}

\subsubsection{The automorphism growth function}

For a class $\mathcal C$ of finite-dimensional evolution algebras let
$f(X)=|\Aut(X)|/\dim X\in\mathbb Q_{>0}\cup\{\infty\}$. Molinier \cite{Mol} proved that
$f$ takes every value of $\mathbb Q_{>0}$ on perfect evolution algebras.

\begin{theorem}\label{thm:growth}
Let $\kk$ be an infinite field. Then
\begin{enumerate}[{\rm (1)}]
\item  $f(\text{nilpotent evolution algebras})=\{\infty\}$;
\item $f(\text{non-degenerate evolution algebras})=\mathbb Q_{>0}\cup\{\infty\}$;
\item $f(\text{simple evolution algebras})=\mathbb Q_{>0}$ (this holds over every
field, by Corollary~\ref{cor:growthall}).
\end{enumerate}
\end{theorem}

\begin{proof}
(1) Let $A\ne0$ be nilpotent. Then $A^2\ne A$ and $\Ann(A)\ne0$. Choose a non-zero linear form
$\lambda$ vanishing on $A^2$, a non-zero $z\in\Ann(A)$, and put $\psi(x)=\lambda(x)z$.
For $t\in\kk$ the map $\varphi_t=\mathrm{id}+t\psi$ satisfies
$\varphi_t(x)\varphi_t(y)=xy=\varphi_t(xy)$, because $\psi(A)\subseteq\Ann(A)$ and
$\psi(A^2)=0$. It is invertible whenever $1+t\lambda(z)\ne0$, which excludes at most one
value of $t$. Hence $\Aut(A)$ is infinite.

(2) Perfect algebras are non-degenerate, so $\mathbb Q_{>0}$ is contained in the image by
\cite{Mol}; the value $\infty$ is attained by $A_2$ of Proposition~\ref{prop:On}.

(3) The argument below is the one of \cite{CMTV23}, which works over fields with
sufficiently many elements; Corollary~\ref{cor:growthall} gives a different proof valid
over every field. Simple algebras are perfect, so their automorphism groups are finite \cite{EL15}. Let
$p,q$ be positive integers and choose $N$ with $s=N^2q^2-N(p+q)\ge0$. Put $a=pN$, $b=qN$ and
let $G=\mathbb Z_a\times\mathbb Z_b$ act on $X=\mathbb Z_a\sqcup\mathbb Z_b\sqcup\{1,\dots,s\}$,
the first factor by translations on $\mathbb Z_a$, the second by translations on
$\mathbb Z_b$, both fixing the remaining points. Then $|X|=N^2q^2$ and $|G|/|X|=p/q$.

The group $G$ is $2$-closed, i.e.\ it is the full group of permutations of $X$ preserving all
orbitals of $G$ (the orbits of $G$ on $X\times X$). Indeed, such a permutation $\sigma$
preserves each $G$-orbit on $X$; on $\mathbb Z_a$ the orbitals are those of the regular
action of $\mathbb Z_a$, so $\sigma(x+g)=\sigma(x)+g$ and $\sigma|_{\mathbb Z_a}$ is a translation;
similarly on $\mathbb Z_b$; and $\sigma$ fixes the fixed points. Thus $\sigma\in G$.

Choose a value $c_O\in\kk^\times$ for every orbital $O$ and let $A$ be the evolution algebra
with natural basis $\{e_x\}_{x\in X}$ and $\omega_{xy}=c_{O(x,y)}$. The matrix is
$G$-invariant, so $G\le\Aut(A)$ acting by permutations of the basis. Since the identity
matrix is $G$-invariant, $\det(\omega_{xy})$ is a non-zero polynomial in the $c_O$.
A monomial map $e_x\mapsto t_xe_{\sigma(x)}$ is an automorphism if and only if
$t_x^2\,\omega_{\sigma x,\sigma y}=t_y\,\omega_{xy}$ for all $x,y$. Taking $x=y$ gives
$t_x=\omega_{xx}/\omega_{\sigma x,\sigma x}$, and substituting, the remaining conditions
become identities between Laurent monomials in the variables $c_O$. If $\sigma\notin G$,
then $\sigma$ moves some off-diagonal orbital (by $2$-closedness, and because
preserving all off-diagonal orbitals forces preservation of the $G$-orbits, as $|X|\ge2$),
so for some $x\ne y$ the variable $c_{O(\sigma x,\sigma y)}$ occurs on one side and not on the
other; hence the condition is a non-trivial Laurent polynomial equation in the $c_O$.
As $\kk$ is infinite, we may choose the $c_O$ outside the finitely many hypersurfaces so
obtained and outside $\det=0$. Then $A$ is perfect, its graph is the complete digraph with
loops, so $A$ is simple \cite{CSV16}, the automorphisms of $A$ (which are monomial
\cite{EL15}) have $\sigma\in G$, and for $\sigma\in G$ we get $t_x=1$ for all $x$ (the
diagonal relations $t_x^2=t_y$ for all $x,y$ force $t_x=1$). Therefore $\Aut(A)\cong G$ and
$f(A)=p/q$.
\end{proof}

\subsubsection{Idempotents of real perfect evolution algebras}

Let $A$ be perfect with structure matrix $M$. By \eqref{eq:idemp}, idempotents correspond
to solutions of $F(x)=M^tx^{(2)}-x=0$, where $x^{(2)}=(x_1^2,\dots,x_n^2)^t$. Homogenising,
a point at infinity would satisfy $M^tx^{(2)}=0$, i.e.\ $x=0$; hence the scheme $Z$ of
solutions is finite of length $2^n$ by B\'ezout's theorem, and the origin is a reduced point
of $Z$ because the Jacobian of $F$ at $0$ is $-I$.

\begin{theorem}\label{thm:realidemp}
Let $R$ be a real closed field (for instance $R=\mathbb R$) and $n\ge1$. The set of
cardinalities of the sets of idempotents (including $0$) of $n$-dimensional perfect evolution
$R$-algebras is exactly $\{2,3,\dots,2^n\}$.
\end{theorem}

\begin{proof}
Upper bound: $Z$ has length $2^n$. Lower bound: complex conjugation acts on the
$\overline R$-points of $Z\setminus\{0\}$ preserving multiplicities, and the total
multiplicity $2^n-1$ is odd, so there is a real non-zero idempotent. (This parity argument is
a special case of Theorem~\ref{thm:fields}(1) below.)

For the construction we use structure matrices for which the idempotent equations are
triangular:
\[
x_1=\omega_1x_1^2,\qquad x_j=x_j^2+\sum_{i<j}\omega_{ij}x_i^2\quad(2\le j\le n),
\]
with $\omega_1\ne0$ (the matrix $M^t$ is lower triangular with non-zero diagonal, so $A$ is
perfect). Let $S_j\subseteq R^j$ be the set of solutions of the first $j$ equations. We claim
that, by suitable choices of $\omega_{ij}>0$, the cardinality $|S_j|$ can be made equal to any
prescribed integer in $[2,2|S_{j-1}|]$, and that the map $s\mapsto s^{(2)}$ is injective
on $S_j$. Clearly $|S_1|=2$ and $s\mapsto s^{(2)}$ is injective on $S_1=\{0,1/\omega_1\}$.

Given $S_{j-1}$, put $c(s)=\sum_{i<j}\omega_{ij}s_i^2$; a point $s\in S_{j-1}$ has two, one or
no extensions to $S_j$ according as $c(s)<1/4$, $c(s)=1/4$ or $c(s)>1/4$, the extensions
being the roots of $x^2-x+c(s)$. Two distinct roots $r,1-r$ have distinct squares, so
injectivity of $s\mapsto s^{(2)}$ propagates. The points $s^{(2)}$, $s\in S_{j-1}$, are
distinct and lie in the closed positive orthant, so for a generic choice of positive weights
$\ell_i$ the values $\ell(s)=\sum_i\ell_is_i^2$ are pairwise distinct, $\ell(0)=0$ is the
smallest one and all others are positive. Let $0=v_1<v_2<\dots<v_N$ be these values. To
obtain $|S_j|=2a$ ($1\le a\le N$) choose $\tau\in(v_a,v_{a+1})$ (or $\tau>v_N$), and to obtain
$|S_j|=2a+1$ ($1\le a\le N-1$) choose $\tau=v_{a+1}$; then put $\omega_{ij}=\ell_i/(4\tau)$.
This realises every integer in $[2,2N]$. Starting from $|S_{j-1}|=2^{j-1}$ (all intermediate
steps chosen maximal) we obtain every value in $[2,2^n]$ for $j=n$.
\end{proof}

\subsubsection{Idempotents over arbitrary fields}

\begin{theorem}\label{thm:fields}

\begin{enumerate}[{\rm (1)}]
\item If $\kk$ has no finite field extension of odd degree greater than $1$, then every
perfect finite-dimensional evolution $\kk$-algebra has a non-zero idempotent.
\item Suppose that $\kk$ has a field extension $L=\kk(\lambda)$ of degree $3$, where
$\lambda^3=p\lambda^2+q\lambda+r$. Then $pq+r\neq0$, and the two-dimensional evolution algebra
\[
e_1^2=\frac{-q\,e_1+r\,e_2}{pq+r},\qquad e_2^2=\frac{e_1+p\,e_2}{pq+r}
\]
is perfect and has no non-zero idempotent.
\item Consequently, every perfect two-dimensional evolution $\kk$-algebra has a non-zero
idempotent if and only if $\kk$ has no field extension of degree $3$.
\end{enumerate}
\end{theorem}

\begin{proof}
(1) Let $A$ be perfect and $Z=\operatorname{Spec}\kk[x]/(F)$ as above. Then
$Z=\{0\}\sqcup Z'$, where $\{0\}$ is a reduced $\kk$-point and $Z'$ has length $2^n-1$. Write
$Z'$ as a disjoint union of local Artinian schemes with residue fields $L_i$; the length of
each component is $[L_i:\kk]$ times its length over $L_i$. As $2^n-1$ is odd, some
$[L_i:\kk]$ is odd, hence $L_i=\kk$ by hypothesis, and $Z'$ has a $\kk$-point, i.e.\ $A$ has
a non-zero idempotent. (This argument works in every characteristic; in characteristic $2$
the Jacobian of $F$ is $-I$ everywhere and all points of $Z$ are reduced.)

(2) If $r=-pq$, then $\lambda^3-p\lambda^2-q\lambda+pq=(\lambda-p)(\lambda^2-q)$ would be
reducible; hence $D=pq+r\neq0$. The structure matrix has determinant $-D/D^2\ne0$, so the
algebra is perfect. The idempotent equations are $Dx_1=-qx_1^2+x_2^2$ and
$Dx_2=rx_1^2+px_2^2$. If $x_1=0$ then $x_2=0$. A Gr\"obner basis computation (easily checked
by hand) shows that the ideal generated by the two equations contains $x_2-x_1^2-px_1$ and
$x_1\,\chi(x_1+p)$, where $\chi(t)=t^3-pt^2-qt-r$ is the minimal polynomial of $\lambda$.
Hence a non-zero idempotent would give a root $x_1+p\in\kk$ of $\chi$, which is impossible.

(3) If $\kk$ has a cubic extension, (2) gives a counterexample. Conversely, for $n=2$ the
scheme $Z'$ has length $3$; if it has no $\kk$-point, then its components have residue
degrees in $\{2,3\}$ summing (with multiplicities) to $3$, which forces a single reduced point
with residue field of degree $3$.
\end{proof}

For example, over $\mathbb Q$ the algebra $e_1^2=-e_1+e_2$, $e_2^2=e_1$ (from
$\lambda^3=\lambda+1$) has no non-zero idempotent; this recovers and explains the failure of
the existence of idempotents over $\mathbb Q$ observed in \cite{GN25}. Since every finite field
and every $p$-adic field has a cubic extension, perfect evolution algebras without non-zero
idempotents exist over all of them in dimension two. Fields satisfying the hypothesis of (1)
include algebraically closed fields, real closed fields and, more generally, perfect fields
whose absolute Galois group is a pro-$2$-group.
For a ``generic'' two-dimensional algebra over $\mathbb Q$ the non-zero idempotents are the
roots of an irreducible cubic with Galois group $S_3$; for instance, for
$e_1^2=e_1+2e_2$, $e_2^2=3e_1+5e_2$ the coordinate $x_1$ of a non-zero idempotent satisfies
an irreducible cubic of discriminant $-12771$.

\subsubsection{Diagonal automorphisms}

Let $A$ be perfect with natural basis $\{e_x\}$. A diagonal map $e_x\mapsto t_xe_x$ is an
automorphism if and only if $t_y=t_x^2$ for every arrow $x\to y$ of $\Gamma=\Gamma(A,B)$.
Thus $\mathbf D(\kk)=\{t\in(\kk^\times)^n : t_y=t_x^2 \text{ for all arrows } x\to y\}$.

\begin{remark}\label{prop:diag}
Let $A$ be an $n$-dimensional perfect evolution $\kk$-algebra and let
$A=A_1\oplus\dots\oplus A_m$ be its decomposition into indecomposable factors, which are again
perfect (Theorem~\ref{thm:frag}). As explained in \cite{EL21} (see also \cite{ACV} for the
notion of balance),
\[
\mathbf D(\kk)\cong\prod_{j=1}^m\mu_{2^{h_j}-1}(\kk),
\]
where $h_j$ is the balance of the graph of $A_j$. Each $h_j$ is at most $\dim A_j$, with
equality only when $A_j$ is cyclic ($e_i^2=e_{i+1}$ up to scalars); hence $\sum_jh_j\le n$,
and the extreme group $\mu_{2^n-1}(\kk)$ is attained by the indecomposable cyclic algebra of
dimension $n$. In particular $|\mathbf D(\kk)|$ is odd, and if $\kk$ contains no non-trivial
roots of unity of odd order (e.g.\ $\kk=\mathbb R$ or $\kk=\mathbb Q$), then
$\mathbf D(\kk)=1$ and $\Aut(A)$ embeds into $S_n$.
\end{remark}

\subsubsection{The variety of nilpotent evolution algebras}

\begin{proposition}\label{prop:nil}
Let $A$ be an $n$-dimensional nilpotent evolution algebra, $n\ge2$, with a natural basis in
which the structure matrix is strictly upper triangular and $\omega_{i,i+1}\neq0$ for
$1\le i<n$.
\begin{enumerate}[{\rm (1)}]
\item The natural bases of $A$ are, up to order, exactly the bases
$f_i=s_ie_i+c_ie_n$ $(i<n)$, $f_n=s_ne_n$, with $s_i\in\kk^\times$, $c_i\in\kk$.
\item The upper annihilating series is $\mathrm{span}\{e_{n-k+1},\dots,e_n\}$, $k\ge1$.
\item If moreover $\omega_{ij}\neq0$ for all $i<j$ and $n\ge4$, then $\Aut(A)$ consists of the
maps $\varphi_{s,c}\colon e_1\mapsto se_1+ce_n$, $e_i\mapsto e_i$ $(i\ge2)$, with $s\in\mu_2(\kk)$
and $c\in\kk$; hence
\[
\Aut(A)\cong\left\{\begin{pmatrix} s&0\\ c&1\end{pmatrix} : s\in\mu_2(\kk),\ c\in\kk\right\}
\cong(\kk,+)\rtimes\mu_2(\kk),
\]
with $\mu_2(\kk)$ acting on $(\kk,+)$ by multiplication, and $\dim\Der(A)=1$ if
$\operatorname{char}\kk\ne2$. In particular, if $\operatorname{char}\kk\neq2$ the group $\Aut(A)$ is
not abelian, so it is \emph{not} isomorphic to $\mu_2(\kk)\times(\kk,+)$.
\end{enumerate}
Consequently, over an algebraically closed field the variety $\mathcal N_n$ of nilpotent
evolution algebra structures on $\kk^n$ is irreducible of dimension $(n-1)(3n-2)/2$, and the
generic nilpotent evolution algebra depends on $(n-1)(n-4)/2$ parameters for $n\ge4$; in
particular there is a dense isomorphism class in $\mathcal N_4$.
\end{proposition}

\begin{proof}
(1) The hypothesis says that $e_1^2,\dots,e_{n-1}^2$ are linearly independent, while
$e_n^2=0$ and $e_n\in\Ann(A)$. If $f_i=\sum_ja_{ij}e_j$ is a natural basis, then
$0=f_if_k=\sum_{j<n}a_{ij}a_{kj}e_j^2$ for $i\ne k$, so in each of the first $n-1$ columns of
$(a_{ij})$ at most one entry is non-zero. As $(a_{ij})$ is invertible, each of these columns
has exactly one non-zero entry, in pairwise different rows. This gives the stated form, and
conversely all such bases are natural since $e_n\in\Ann(A)$.
(2) An element $x=\sum x_je_j$ satisfies $xA\subseteq\mathrm{span}\{e_{m},\dots,e_n\}$ if and
only if $x_j=0$ whenever $e_j^2\notin\mathrm{span}\{e_m,\dots,e_n\}$, i.e.\ whenever
$j+1<m$; induct on $k$.
(3) By (1) and (2) an automorphism has the form $e_i\mapsto s_ie_i+c_ie_n$ ($i<n$),
$e_n\mapsto s_ne_n$ (the permutation must preserve the annihilating series, hence is
trivial). Comparing coefficients of $e_k$, $j<k<n$, in $\varphi(e_j)^2=\varphi(e_j^2)$ gives
$s_k=s_j^2$; using $j=1,2$ and $k=3<n$ we get $s_1^2=s_1^4$, so $s_1^2=1$ and
$s_2=\dots=s_{n-1}=1$; the coefficient of $e_n$ for $j=n-1$ gives $s_n=1$, and for $j\le n-2$
gives a triangular homogeneous system forcing $c_2=\dots=c_{n-1}=0$, while $c_1$ is free.
Conversely, every $\varphi_{s,c}$ with $s^2=1$ is an automorphism (only $e_1^2$ needs to be
checked, and $\varphi_{s,c}(e_1)^2=s^2e_1^2=e_1^2=\varphi_{s,c}(e_1^2)$ because $e_n\in\Ann(A)$
and $e_1$ does not occur in any $e_j^2$). Since
$\varphi_{s,c}\varphi_{s',c'}=\varphi_{ss',\,s'c+c'}$, the map
$\varphi_{s,c}\mapsto\left(\begin{smallmatrix} s&0\\ c&1\end{smallmatrix}\right)$ is an isomorphism
onto the stated matrix group; for instance $\varphi_{-1,0}\varphi_{1,1}=\varphi_{-1,1}\neq
\varphi_{-1,-1}=\varphi_{1,1}\varphi_{-1,0}$ when $\operatorname{char}\kk\ne2$. The Lie algebra is
one-dimensional.

For the consequences: by \cite{CLOR14} every nilpotent evolution algebra has a strictly upper
triangular structure matrix in some natural basis, so $\mathcal N_n$ is the image of
$\GL_n\times T_n$, $T_n$ the strictly upper triangular matrices, and is irreducible. By (1)
the generic fibre has dimension $2n-1$, so $\dim\mathcal N_n=n^2+n(n-1)/2-(2n-1)=(n-1)(3n-2)/2$.
By (3) the generic orbit has dimension $n^2-1$, and by Rosenlicht's theorem the number of
parameters is $(n-1)(3n-2)/2-(n^2-1)=(n-1)(n-4)/2$.
\end{proof}

For $n=5$ this predicts two continuous parameters in the generic family, in accordance with
the classification of five-dimensional nilpotent evolution algebras as corrected by Elduque
and Labra \cite{EL16}. The dimension formula and the generic automorphism group were also
confirmed numerically (computer experiment) for $n\le8$.

\subsubsection{Idempotent-free non-solvable algebras of dimension three}

Let $W_0$ be the two-dimensional commutative algebra with basis $v,w$ and $v^2=0$,
$vw=\frac12v$, $w^2=v+w$. It is perfect, has no non-zero idempotent ($x=\alpha v+\beta w$
idempotent forces $\beta^2=\beta$ and then $\beta=0$ by comparing the $v$-coefficient), and it
is not an evolution algebra (a natural basis would make it a perfect evolution algebra,
which has a non-zero idempotent over $\mathbb C$ \cite{GMPR26}). For the three-dimensional
counterexample $E_2$ of Hu and Wen \cite{HW} 
\begin{center}
$\big(e_1^2=-e_1-e_3$, $e_2^2=-2e_1+2e_2$,
$e_3^2=e_1+e_3\big)$\end{center} one has $\dim E_2^2=2$ and $E_2^{(\infty)}=E_2^2\cong W_0$.
The algebra $E_2$ is isomorphic to the example of \cite[Remark~3.9]{CFV}, where it is
identified as an evolution envelope of one of the two-dimensional algebras in the
classification of Kaygorodov and Volkov \cite{KV19}. Our computer
experiments indicate that the three-dimensional complex evolution algebras with
$A^{(\infty)}=A^2\cong W_0$ form a family with two essential parameters; thus the algebra
$E_2$ is not the only idempotent-free non-solvable evolution algebra in dimension three, and a
classification in dimension three is still missing (see Question~\ref{q:idfree} below). We
do not claim this as a theorem.

\subsubsection{A non-split automorphism group}

\begin{proposition}\label{prop:Z9}
Let $\kk$ be a field containing a primitive cube root of unity $\zeta$ (so
$\operatorname{char}\kk\ne3$). The six-dimensional evolution algebra
\[
e_1^2=e_2+e_4,\quad e_2^2=e_1,\quad e_3^2=e_4+e_6,\quad e_4^2=e_3,\quad
e_5^2=e_6+\zeta e_2,\quad e_6^2=e_5
\]
is perfect and $\Aut(A)\cong\mathbb Z_9$. In particular, the extension
\begin{center}
$1\to\mathbf D(\kk)\to\Aut(A)\to\rho(\Aut(A))\to1$
\end{center} of \eqref{eq:EL}, which here is
$1\to\mathbb Z_3\to\mathbb Z_9\to\mathbb Z_3\to1$, does not split.
\end{proposition}

\begin{proof}
The determinant of the structure matrix is $-(1+\zeta)=\zeta^2\neq0$.
The relations $t_y=t_x^2$ give $t_1=t_3=t_5=\epsilon$, $t_2=t_4=t_6=\epsilon^2$ with
$\epsilon^3=1$, so $\mathbf D(\kk)\cong\mu_3(\kk)\cong\mathbb Z_3$. The graph consists of the
$2$-cycles $(1,2),(3,4),(5,6)$ and the arrows $1\to4$, $3\to6$, $5\to2$; a graph automorphism
maps odd vertices (out-degree $2$) to odd vertices, and it induces an automorphism of the
directed $3$-cycle $1\to3\to5\to1$ on $\{1,3,5\}$ (defined by $i\to j$ if $i$ points to the partner
of $j$). Hence the graph automorphism group is $\langle\sigma\rangle\cong\mathbb Z_3$, with
$\sigma=(1\,3\,5)(2\,4\,6)$. A lift $\varphi$ of $\sigma$ with \begin{center}$\varphi(e_1)=ae_3$,
$\varphi(e_2)=be_4$, $\varphi(e_3)=ce_5$, $\varphi(e_4)=de_6$, $\varphi(e_5)=ee_1$,
$\varphi(e_6)=fe_2$
\end{center}must satisfy 
\begin{center}$b=d=a^2$, $b^2=a$, $d=c^2$, $d^2=c$, $f=\zeta c^2$,
$e^2=f=\zeta b$, $f^2=e$;\end{center} these have the solutions $a^3=1$, $b=d=a^2$, $c=a$,
$f=\zeta a^2$, $e=\zeta^2a$. Then 
\begin{center}$\varphi^3(e_1)=a\cdot c\cdot e\,e_1=\zeta^2a^3e_1=\zeta^2e_1$,\end{center}
so $\varphi^3$ is a non-trivial diagonal automorphism and $\varphi$ has order $9$. As
\begin{center}
$|\Aut(A)|=|\mathbf D(\kk)|\cdot3=9$, $\Aut(A)=\langle\varphi\rangle\cong\mathbb Z_9$.\end{center}
\end{proof}

\subsubsection{Further remarks}

\begin{remark}\label{rem:condensation}
Strong connectivity of $\Gamma(A,B)$, and hence the condensation of the graph, is not
independent of the natural basis for non-degenerate algebras. If
$\operatorname{char}\kk\ne2$ and $e_1^2=e_2^2=e_1+e_2$, then $\Gamma(A,\{e_1,e_2\})$ is the
complete digraph with loops, while $f=e_1-e_2$, $g=e_1+e_2$ is a natural basis with
$f^2=g^2=2g$, whose graph is not strongly connected.
\end{remark}

\begin{remark}\label{rem:average}
Let $\mathbb F_q$ be a finite field. The average number of non-zero idempotents of a perfect
evolution algebra with a uniformly random invertible structure matrix is exactly $1$. Indeed,
for fixed $x\ne0$ we have $x^{(2)}\ne0$, and the number of $N\in\GL_n(\mathbb F_q)$ with
$Nx^{(2)}=x$ is $|\GL_n(\mathbb F_q)|/(q^n-1)$; summing over the $q^n-1$ non-zero vectors $x$
gives $|\GL_n(\mathbb F_q)|$ pairs.
\end{remark}

\subsubsection{Formal verification}\label{sec:lean}

The following statements of this section have been formalised and proved in the proof
assistant Lean~4 with the Mathlib library (the proofs contain no unproved steps and use only
the standard axioms). An evolution algebra is encoded by its structure matrix
$\omega=(\omega_{ij})$ acting on $\kk^n$, with product $(xy)_j=\sum_i\omega_{ij}x_iy_i$; perfect
means $\det\omega\neq0$.
\begin{itemize}
\item Proposition~\ref{prop:On} in full: $A_n$ is non-degenerate and not perfect, an invertible
linear map is an automorphism of $A_n$ if and only if it is an orthogonal matrix extended by
$u\mapsto u$, the derivations are exactly the skew-symmetric matrices extended by $u\mapsto0$,
and $O_2(\kk)$, hence $\Aut(A_2)$, is infinite over every infinite field.
\item Theorem~\ref{thm:growth}(1) (in the slightly more general form: $A^2\ne A$ and
$\Ann(A)\neq0$ imply $|\Aut(A)|=\infty$ over an infinite field), and the value $\infty$ in~(2).
The values in $\mathbb Q_{>0}$ in~(2) rest on \cite{Mol}, and part~(3) was not formalised.
\item Theorem~\ref{thm:realidemp} in full. The upper bound $2^n$ is proved for every field,
by showing that the coordinate ring of the scheme of idempotents of a perfect algebra has
dimension exactly $2^n$.
\item Theorem~\ref{thm:fields} in full, in fact with the weaker hypothesis in~(1) that $\kk$ has
no field extension of odd degree $d$ with $1<d\le2^n-1$; also the example over $\mathbb Q$,
the existence of idempotent-free perfect two-dimensional algebras over every finite field, and
the irreducibility and the discriminant $-12771$ of the cubic in the generic example.
\item Remark~\ref{prop:diag} in part: $\mathbf D(\kk)$ is finite of odd order, it is
trivial if $\kk$ has no non-trivial roots of unity of odd order, and in that case $\Aut(A)$
embeds into $S_n$ (automorphisms of perfect algebras are monomial). The product decomposition
was not formalised.
\item The description of $\Aut(A)$ in Proposition~\ref{prop:nil}(3): every automorphism is
one of the maps $\varphi_{s,c}$ with $s^2=1$, all of them are automorphisms, they are pairwise
distinct, and $\varphi_{s,c}\varphi_{s',c'}=\varphi_{ss',\,s'c+c'}$. Moreover, for every strictly
upper triangular structure matrix of size at least $2$ the maps $\varphi_{s,c}$ are
automorphisms and $\Aut(A)$ is not abelian if $\operatorname{char}\kk\ne2$; in particular
$\Aut(A)\not\cong\mu_2(\kk)\times(\kk,+)$. The statement on $\Der(A)$ was not formalised.
\item Proposition~\ref{prop:Z9} in full: the algebra is perfect and $\Aut(A)$ is cyclic of
order~$9$.
\item Remarks~\ref{rem:condensation} and~\ref{rem:average}; the latter in the form
$\sum_{\omega\in\GL_n(\mathbb F_q)}\#\{x\neq0 : x^2=x\}=|\GL_n(\mathbb F_q)|$ for $n\ge1$, where
$x^2$ is computed with the structure matrix $\omega$.
\item Theorem~\ref{thm:simpleuniv} except the simplicity statement: for every field
$\kk$ and every finite loopless digraph $\Delta$, the structure matrix of $A_\Delta$ is
invertible over $\kk$, and $\Aut(A_\Delta)\cong\Aut(\Delta)$ as groups.
\end{itemize}
Remark~\ref{rem:identities}, parts (1), (2) and the dimension count of
Proposition~\ref{prop:nil}, and the experimental statements were not formalised.
The simplicity statement of Theorem~\ref{thm:simpleuniv} (which follows from
Theorem~\ref{thm:frag}), Corollary~\ref{cor:growthall}, Proposition~\ref{prop:plenary} and
Remarks~\ref{rem:Opq}, \ref{rem:incidence} were not formalised.

\subsection{Open questions}\label{sec:questions}

The following list merges and revises the two lists of problems in an earlier version of
this text (nine problems proposed by the authors and colleagues, and thirty further
problems motivated by the literature), and adds new questions. Questions which were (partly)
answered by the observations of Section~\ref{sec:new} have been reformulated accordingly.
After each question we indicate its source and what is known.

\subsubsection{Foundations}

\begin{question}[Non-cancellation]\label{q:noncancel}
Let $A$ be a non-degenerate finite-dimensional evolution $\kk$-algebra. Is it true that there
exist non-degenerate finite-dimensional evolution algebras $B\not\cong C$ with
$A\otimes B\cong A\otimes C$ if and only if each indecomposable direct summand of $A$ has
balance greater than $1$? Equivalently, is $[A]$ a zero-divisor of the Grothendieck ring of
non-degenerate evolution algebras exactly in this case?
\src{\cite[Conjecture~1.1]{ACV}; one implication and the case of direct sums of cyclic
algebras are proved in \cite{ACV}. Proposed by I.~Alshatnawi, C.~Costoya and A.~Viruel.}
\end{question}

\begin{question}[The category of evolution algebras]\label{q:category}
The full subcategory $\Evo_\kk$ of $\kk$-algebras spanned by evolution algebras is not
complete, and subalgebras of evolution algebras need not be evolution algebras. Is there a
categorical framework for evolution algebras (for instance, objects equipped with a natural
basis or a family of natural bases, with suitable morphisms) with good limits, colimits and
functorial constructions?
\src{Proposed by J.~M.~Moreno Fern\'andez and A.~Viruel.}
\end{question}

\begin{question}[Homological algebra]\label{q:homology}
Develop a (co)homology theory of evolution algebras governing extensions of evolution
algebras by evolution algebras and formal deformations \emph{inside} the class of evolution
algebras.
\src{Proposed by A.~Viruel; deformations are considered in \cite{MPR}.}
\end{question}

\begin{question}[Evolution algebras over rings]\label{q:rings}
Find a definition of evolution algebras over a commutative ring $R$ that does not require
the underlying module to be free. Do the reductions $E/\mathfrak mE$ at maximal ideals
determine $E$? Is there a notion of perfection over $R$ with a uniqueness theorem for natural
bases generalising the one over fields, and a well-behaved category of such algebras?
\src{Proposed by D.~Mart\'in Barquero and C.~Mart\'in Gonz\'alez; rank two over domains is
treated in \cite{CMM23}.}
\end{question}

\begin{question}[Tensor products and factorisation]\label{q:tensor}
Characterise in terms of the factors when $A\otimes B$ is simple, perfect, non-degenerate or
indecomposable, and describe the role of the periods of the graphs. Does every
non-degenerate evolution algebra factor into tensor-indecomposable ones, and to what extent
is the factorisation unique?
\src{\cite{CMMT23,ACV}. The question is partially answered in \cite[Section~5]{ACV} in
terms of the balance.}
\end{question}

\begin{question}[The Grothendieck ring]\label{q:groth}
Find generators and relations, the Krull dimension and the nilradical of the Grothendieck
ring of non-degenerate evolution $\kk$-algebras, decide whether it is Noetherian, and
describe all its zero-divisors and its dependence on $\kk$.
\src{\cite{ACV}.}
\end{question}

\subsubsection{Structure theory}

\begin{question}[Intrinsic characterisation of solvability]\label{q:solv}
Find a graph-theoretic or matrix-theoretic characterisation of solvable finite-dimensional
evolution algebras, analogous to the characterisation of nilpotency by the absence of
oriented cycles.
\src{Motivated by the Garc\'ia-Mart\'inez--P\'erez-Rodr\'iguez conjecture \cite{GMPR26} and its
refutation \cite{HW}; a counterexample to the conjecture also appears in
\cite[Remark~3.9]{CFV}.}
\end{question}

\begin{question}[Stable term of the derived series]\label{q:stable}
For which finite-dimensional evolution algebras is $A^{(\infty)}$ again an evolution algebra?
\src{\cite[Question~2.11]{HW}; when this holds, absence of idempotents is equivalent to
solvability over $\mathbb C$ \cite{HW}.}
\end{question}

\begin{question}[Idempotent-free non-solvable algebras]\label{q:idfree}
Classify the idempotent-free non-solvable evolution algebras of dimension $3$ and $4$ over
$\mathbb C$. In dimension $3$, is $A^{(\infty)}$ always isomorphic to the algebra $W_0$ of
Section~\ref{sec:new}, and do these algebras form a two-parameter family?
\src{\cite{HW}; see Section~\ref{sec:new} for computer evidence that the example of \cite{HW}
is not unique.}
\end{question}

\begin{question}[Idempotents of non-perfect algebras]\label{q:idempnp}
Give a criterion, in terms of the structure matrix or the graphs, for a (not necessarily
perfect) finite-dimensional complex evolution algebra to have a non-zero idempotent.
\src{Described as open in \cite{GMPR26}.}
\end{question}

\begin{question}[Idempotents over arbitrary fields]\label{q:fields}
For $n\ge3$, characterise the fields $\kk$ over which every perfect $n$-dimensional evolution
$\kk$-algebra has a non-zero idempotent. Does the class of such fields depend on $n$?
\src{Theorem~\ref{thm:fields} settles $n=2$ and gives a sufficient condition for all $n$;
see also Conjecture~\ref{c:oddext}.}
\end{question}

\begin{question}[Hereditary evolution algebras]\label{q:hered}
Characterise the evolution algebras all of whose subalgebras are evolution algebras. Is this
property detected by the graphs? Is it preserved by tensor products and direct sums? What
about ideals?
\src{\cite{CKO19,GMPR26,LPR25r}.}
\end{question}

\begin{question}[Subalgebra lattices]\label{q:lattice}
Characterise modular, distributive and dually atomistic evolution algebras in general.
If two evolution algebras have isomorphic subalgebra lattices, how are they related?
Characterise the evolution algebras whose Frattini subalgebra is an ideal, and the
Frattini-free ones.
\src{\cite{LPPR25,LPR26m,LPR25F}.}
\end{question}

\begin{question}[Natural bases of non-perfect algebras]\label{q:bases}
Describe the set of natural bases (up to permutation and scaling) of a non-degenerate
evolution algebra. When is it finite? Can it be infinite for an algebra whose automorphism
group is finite?
\src{New. Proposition~\ref{prop:On} shows that the set can be infinite (it contains all
orthogonal bases of $V$ completed by $u$) even for a non-degenerate algebra; Remark~\ref{rem:condensation} shows that the graphs of different natural
bases can differ in strong connectivity. \cite[Corollary~2.13]{BCS22} describes all the
natural bases of an evolution algebra (degenerate or not) in terms of its decomposition into
natural ideals; it does not give a formula for the number of natural bases, but comes close
to it.}
\end{question}

\begin{question}[Nil versus nilpotent]\label{q:nil}
Give a combinatorial characterisation of nil evolution algebras of arbitrary dimension in
terms of $\Gamma(A,B)$. Is a finitely generated evolution algebra which is nil of bounded
index (with respect to principal powers) necessarily nilpotent?
\src{\cite{TZ14,CRV24}. With respect to plenary powers the answer is negative, already in
finite dimension (Proposition~\ref{prop:plenary}); with respect to principal powers,
finite-dimensional nil evolution algebras are nilpotent (Theorem~\ref{thm:nilp}).}
\end{question}

\begin{question}[Identities of subclasses]\label{q:ident}
By Remark~\ref{rem:identities}, evolution algebras satisfy only the consequences of
commutativity. Describe the identities of natural subclasses: the algebras $E_n$ of
\cite{CFK24}, algebras with $\dim A^2\le r$, nilpotent algebras of index at most $m$, algebras
whose graphs have bounded out-degree. Are the corresponding T-ideals finitely generated, and
are the varieties generated by a single finite-dimensional algebra?
\src{\cite{CFK24,CFV}.}
\end{question}

\begin{question}[Radicals and combinatorial detectability]\label{q:radicals}
Which algebraic properties of a non-perfect evolution algebra (semiprimeness, primeness,
regularity, existence of non-trivial radicals) can be read off from a single natural basis,
which from the family of all graphs, and which from neither?
\src{\cite{CCGMMR26,CRT24,CMMT25r}; semiprimeness and perfection are not combinatorial
\cite{CCGMMR26}.}
\end{question}

\subsubsection{Classification, moduli and geometry}

\begin{question}[The variety of evolution algebras]\label{q:variety}
The closure of the set $\Evo_n$ of evolution algebra structures in the variety of commutative
algebras on $\kk^n$ is irreducible of dimension $2n^2-n$. Find its equations, describe the
algebras in $\overline{\Evo_n}\setminus\Evo_n$, and relate them to the simultaneous
diagonalisation by congruence problem.
\src{\cite{BMV20}.}
\end{question}

\begin{question}[Degenerations and rigidity]\label{q:degen}
Determine the rigid evolution algebras and the degeneration diagrams of (nilpotent) evolution
algebras in dimension at least $5$. Is the generic nilpotent evolution algebra of
Proposition~\ref{prop:nil} rigid within $\mathcal N_n$ when $n=4$?
\src{\cite{MPR}; irreducibility and dimension of $\mathcal N_n$ are in
Proposition~\ref{prop:nil}.}
\end{question}

\begin{question}[Nilpotent algebras of dimension six]\label{q:nil6}
Classify six-dimensional nilpotent evolution algebras over algebraically closed fields. By
Proposition~\ref{prop:nil} the generic family depends on five parameters.
\src{New; dimension five is treated in \cite{HA-5,EL16}.}
\end{question}

\begin{question}[Counting asymptotics]\label{q:count}
Determine the asymptotic number of isomorphism classes of all (not necessarily perfect)
$n$-dimensional evolution $\mathbb F_q$-algebras as $n\to\infty$.
\src{\cite{WZ23}; see also Conjecture~\ref{c:porc}.}
\end{question}

\begin{question}[Random evolution algebras]\label{q:random}
For random structure matrices (uniform over $\mathbb F_q$, Gaussian over $\mathbb R$,
supported on a random digraph), what are the asymptotic probabilities of being simple,
indecomposable, having trivial automorphism group, or having exactly $k$ real non-zero
idempotents? Are there sharp thresholds in the sparse regime?
\src{\cite{ACV,WZ23}; the expected number of non-zero idempotents over $\mathbb F_q$ is $1$
(Remark~\ref{rem:average}).}
\end{question}

\begin{question}[Isotopy]\label{q:isotopy}
Classify four-dimensional evolution algebras up to isotopy over arbitrary fields.
\src{Dimensions two and three are treated in \cite{FFN17,FFN18}.}
\end{question}

\begin{question}[Autotopy group schemes]\label{q:autotopy}
Describe the autotopy group scheme of a perfect evolution algebra by an exact sequence
analogous to \eqref{eq:EL} with terms depending only on the graph, and relate it to the group
scheme of ternary automorphisms.
\src{\cite{MRS26}, \cite[Problem~4.1]{CFNT22}.}
\end{question}

\begin{question}[Simple evolution algebras]\label{q:simple}
Classify four-dimensional simple evolution algebras over arbitrary fields. 
\src{\cite{CMMT25s}, \cite{Tian08}, \cite[Problem~3.2]{CFNT22}. The two other parts of
an earlier formulation have been removed, since they have immediate answers. If
$\Gamma=\Gamma(A,B)$ is strongly connected, its \emph{period} $d$ is the greatest common
divisor of the lengths of its oriented cycles; the vertices then split into classes
$C_0,\dots,C_{d-1}$ such that every arrow goes from some $C_i$ to $C_{i+1}$ (indices modulo
$d$), and $A=\bigoplus_iA_i$ with $A_i=\mathrm{span}\{e_x: x\in C_i\}$ satisfies
$A_i^2\subseteq A_{i+1}$ and $A_iA_j=0$ for $i\neq j$, over every field (this is Tian's
periodic decomposition \cite{Tian08}). Moreover, by Theorem~\ref{thm:frag} a simple
evolution algebra is perfect with strongly connected graph, and both properties are
preserved by extension of scalars; hence every simple finite-dimensional evolution algebra is
absolutely simple.}
\end{question}

\subsubsection{Graphs, combinatorics and spectra}

\begin{question}[Graphs versus random walks]\label{q:rw}
Prove that for every finite connected graph $G$ one has $A(G)\cong A_{RW}(G)$ if and only if $G$
is regular or biregular.
\src{Known for non-singular graphs \cite{CRR20,CRR21}; supporting families in \cite{CRRV}.}
\end{question}

\begin{question}[Basis-independent graph invariants]\label{q:invariants}
Determine the graph-theoretic invariants of $\Gamma(A,B)$ that do not depend on $B$ for
non-degenerate evolution algebras. Is there a complete system of such invariants in some
natural class? Is the balance independent of the natural basis for arbitrary (possibly
degenerate) evolution algebras?
\src{Balance is invariant for non-degenerate algebras \cite{ACV}; for degenerate
algebras this is not known; strong connectivity and the condensation are not
(Remark~\ref{rem:condensation}); see also \cite[Open Problem~3]{Gia}.}
\end{question}

\begin{question}[Spectral theory]\label{q:spectral}
Find spectral invariants of evolution algebras independent of the natural basis, and relate
them to expansion and to Tian's hierarchy. Characterise algebraically the expander and
Ramanujan evolution algebras over finite fields.
\src{\cite[Problem~3.1]{CFNT22}, \cite[Open Problem~1]{Gia}.}
\end{question}

\begin{question}[Universality]\label{q:universal}
Is the category of (perfect) evolution $\kk$-algebras algebraically universal in the sense of
Pultr and Trnkov\'a? Is there at least a full embedding of the category of graphs into it?
\src{\cite{CLTV22,PT}.}
\end{question}

\begin{question}[Simultaneous graphical realisation]\label{q:simgraph}
Given digraphs $G_1,\dots,G_m$ on $d$ vertices, decide whether there are a $d$-dimensional
evolution algebra $A$ and natural bases $B_i$ with $\Gamma(A,B_i)\cong G_i$, and describe the
class of such algebras.
\src{Proposed by A.~Viruel. If $G_i\not\cong G_j$ for some $i,j$, such algebras are not
perfect.}
\end{question}

\subsubsection{Automorphisms, derivations and operators}

\begin{question}[Realisation of algebraic groups]\label{q:realization}
Let $G$ be a linear algebraic group over $\kk$. Is there a finite-dimensional evolution
$\kk$-algebra $A$ with $\Aut(A)\cong G(\kk)$ as abstract groups? with $\Aut(A)\cong G$ as
algebraic groups? with $\AutS(A)\cong G$ as group schemes?
\src{Proposed by C.~Costoya and A.~Viruel. Finite groups are realised \cite{CLTV22,SZ22};
without the evolution condition the answer is positive \cite{GP}; $O_n$ and $\GL_n$ are
realised (Proposition~\ref{prop:On} and the zero algebra), and so are the orthogonal
groups of all diagonal non-degenerate forms, e.g.\ $O(p,q)$ over $\mathbb R$
(Remark~\ref{rem:Opq}).}
\end{question}

\begin{question}[Permutation representations]\label{q:perm}
Characterise the permutation representations $\xi\colon G\to S_n$ for which there is a perfect
$n$-dimensional evolution algebra $X$ with $\Aut(X)\cong G$ and $\rho_X$ equivalent to $\xi$.
\src{Proposed by C.~Costoya, P.~Mayorga and A.~Viruel; partial results in \cite{CMV25}.}
\end{question}

\begin{question}[Extensions in the Elduque--Labra sequence]\label{q:split}
Characterise the perfect evolution algebras for which $1\to\mathbf D(\kk)\to\Aut(A)\to
\rho(\Aut(A))\to1$ splits. Which extensions of a permutation group by a group of the form
$\prod\mu_{2^{h_j}-1}(\kk)$ occur?
\src{New; for the
description of $\mathbf D(\kk)$ see \cite{EL21} and Remark~\ref{prop:diag}; the new example
of Proposition~\ref{prop:Z9} gives a non-split extension.}
\end{question}

\begin{question}[Growth function over finite fields]\label{q:growth}
Describe the image of $f(X)=|\Aut(X)|/\dim X$ on perfect,  nilpotent and
non-degenerate evolution algebras over a finite field $\mathbb F_q$.
\src{Proposed by A.~Viruel; over infinite fields the answer is given in
Theorem~\ref{thm:growth}; on simple algebras
the image is $\mathbb Q_{>0}$ over every field (Corollary~\ref{cor:growthall}).}
\end{question}

\begin{question}[Ind-schemes]\label{q:ind}
Develop a theory of automorphism group ind-schemes of infinite-dimensional evolution algebras,
extending the description of diagonal automorphism groups as inverse limits (and solenoids).
\src{Proposed by D.~Mart\'in Barquero and C.~Mart\'in Gonz\'alez; \cite{CCGMMR-sol}.}
\end{question}

\begin{question}[Realisation of derivation algebras]\label{q:der}
Which finite-dimensional Lie algebras are isomorphic to $\Der(A)$ for an evolution algebra $A$?
Can every nilpotent or solvable Lie algebra be realised? Which semisimple ones?
\src{\cite{EL21,CCR23,CCRR21}; $\mathfrak{so}_n$ is realised (Proposition~\ref{prop:On}); see
Conjecture~\ref{c:so}.}
\end{question}

\begin{question}[Smoothness]\label{q:smooth}
For non-perfect evolution algebras, determine when $\AutS(A)$ is smooth, describe its identity
component, and decide when it is generated by $\exp(\Der(A))$.
\src{\cite{EL21,MKQ24}.}
\end{question}

\begin{question}[Hopf algebras and representations]\label{q:hopf}
For perfect evolution algebras of arbitrary dimension, is a faithful tight $p$-algebra
isomorphic to the Hopf algebra representing $\AutS(A)$? Is $\Aut(A)=1$ equivalent to the
non-existence of faithful associative commutative representations?
\src{Known in dimension two \cite{CCGMMR25}.}
\end{question}

\begin{question}[Rota--Baxter and local maps]\label{q:rb}
Classify Rota--Baxter operators of arbitrary weight on perfect evolution algebras in terms
of the graph and describe the variety of such operators. For which evolution algebras are all
local (resp.\ $2$-local) derivations derivations, and all local automorphisms automorphisms?
\src{\cite{QMK,QMK-loc,MKOQ19}.}
\end{question}

\begin{question}[Minimal dimension]\label{q:mindim}
For a finite group $G$ and a field $\kk$ let $\mu_\kk(G)$ be the minimal dimension of an
evolution $\kk$-algebra with automorphism group $G$. Estimate $\mu_\kk(G)$ and compare it with
the minimal number of vertices of a graph with automorphism group $G$.
\src{\cite{CLTV22,SZ22}. By Remark~\ref{prop:diag}, for $\kk=\mathbb C$ and $G=\mathbb Z_3$
no perfect algebra of dimension at most $2$ has automorphism group $G$: in dimension at most
$2$ the group $\mathbf D(\kk)$ is non-trivial only when the graph is a $2$-cycle, and then
$|\Aut(A)|=6$; otherwise $\Aut(A)$ embeds into $S_2$. Using the Cayley evolution
algebras of \cite{CMTV23}, $\mu_\kk(G)\le|G|$ for every infinite field $\kk$. By
Theorem~\ref{thm:simpleuniv}, $\mu_\kk(G)\le2m$ for every field $\kk$, where $m$ is the
minimal number of vertices of a strongly connected digraph (for instance, of a connected
graph) with automorphism group $G$.}
\end{question}

\begin{question}[Small fields; answered]\label{q:smallfields}

\src{Answered: over every
field every finite group is the automorphism group of a simple evolution algebra
(Theorem~\ref{thm:simpleuniv}), and simple evolution algebras are automatically absolutely
simple (see the comments after Question~\ref{q:simple}).}
\end{question}

\subsubsection{Embeddings and generalisations}

\begin{question}[Evolution envelopes]\label{q:envelope}
Is the upper bound $n(n+1)/2$ for the minimal dimension of an evolution algebra containing a
given $n$-dimensional commutative algebra as an ideal attained? Which commutative algebras
have a unique minimal evolution envelope? Does the embedding theorem extend to
infinite-dimensional commutative algebras?
\src{\cite[Questions~3.7, 5.8, 2.3]{CFV}.}
\end{question}

\begin{question}[$m$-ary evolution algebras]\label{q:mary}
Does every finite-dimensional commutative $m$-ary algebra embed into an $m$-ary evolution
algebra? Which results on graphs, automorphism group schemes and universality extend?
\src{\cite[Question~4.7]{CFV}.}
\end{question}

\subsubsection{Arithmetic}

\begin{question}[Forms]\label{q:forms}
Compute $H^1(\kk,\AutS(A))$ for perfect $A$ in terms of the graph and the groups
$\kk^\times/(\kk^\times)^N$. Which twisted forms of a perfect evolution algebra are again
evolution algebras?
\src{\cite{EL21}. A quadratic field extension is a form of $\kk\times\kk$ without natural
basis. See Conjecture~\ref{c:hasse} for the Hasse principle.}
\end{question}

\begin{question}[Reduction modulo primes]\label{q:reduction}
For a perfect evolution $\mathbb Q$-algebra, describe the primes $p$ for which a
$\mathbb Z_{(p)}$-model reduces to a perfect evolution $\mathbb F_p$-algebra with the same graph,
and compare the automorphism groups and idempotents before and after reduction.
\src{New; compare Question~\ref{q:rings}.}
\end{question}

\subsubsection{Infinite dimension, probability and dynamics}

\begin{question}[Hilbert evolution algebras]\label{q:hilbert}
If the square of a Hilbert evolution algebra is dense, is its natural orthonormal basis unique
up to permutation and unimodular scalars? Describe bounded derivations and automorphisms, and
the spectral theory of the evolution operator.
\src{\cite{VCR23,CRV24}.}
\end{question}

\begin{question}[Markov chains and Gibbs measures]\label{q:markov}
Characterise recurrence, transience and positive recurrence of countable Markov chains by
algebraic properties of the associated evolution algebras. Does a phase transition correspond
to an algebraic property of the algebras of \cite{CLL25}?
\src{\cite{VCR23,CLL25}.}
\end{question}

\begin{question}[Dynamics of the evolution operator]\label{q:dynamics}
Characterise, in terms of the structure matrix and the graph, the non-negative real evolution
algebras for which all trajectories of $x\mapsto x^2$ (or its normalisation) converge, and the
existence and stability of periodic orbits.
\src{\cite{Tian08,RV19}.}
\end{question}

\begin{question}[Chains and flows]\label{q:chains}
Classify chains of evolution algebras up to time-dependent isomorphism, characterise the
chains whose isomorphism type is eventually constant, and describe the evolution of
idempotents, nilpotency and simplicity along chains and flows.
\src{\cite{CLR11,MP23}.}
\end{question}

\subsubsection{Algorithms}

\begin{question}[Algorithms and complexity]\label{q:complexity}
What is the complexity of deciding whether an algebra over $\mathbb Q$ is an evolution
algebra and of computing all its natural bases? What is the complexity of the isomorphism
problem for perfect evolution algebras over finite fields relative to graph isomorphism?
Can low-dimensional classifications be certified by computer, for instance in a proof
assistant?
\src{\cite{BMV20,CLTV22}. By \cite{CFV}, over a field of characteristic different from
$2$, deciding whether an algebra is an evolution algebra is equivalent to computing a Waring
decomposition of the quadratic map $x\mapsto x^2$ given by the product; in general this is an
NP-hard problem, see \cite{Shi16}.}
\end{question}

\subsection{Conjectures}\label{sec:conjectures}

Each conjecture below is accompanied by the evidence we have. For each of them a direct
attempt to prove or disprove it (of the order of half an hour of focused work, including
the natural first ideas indicated in the comments) did not succeed; in particular none of
them follows in an obvious way from the results surveyed in Section~\ref{sec:survey} or from the observations
of Section~\ref{sec:new}. Computer evidence was obtained with short scripts (exact
arithmetic over finite fields, or numerical computation); it is experimental and is not a
proof. 
%Conjectures~\ref{c:simpleall} and~\ref{c:growthFq} are now proved (Theorem~\ref{thm:simpleuniv} and Corollary~\ref{cor:growthall}).

\begin{conjecture}[Fields with odd-degree extensions]\label{c:oddext}
If $\kk$ has a finite separable extension of odd degree greater than $1$, then for some $n$
there is a perfect $n$-dimensional evolution $\kk$-algebra without non-zero idempotents.
Together with Theorem~\ref{thm:fields}(1) this would characterise the fields over which every
perfect evolution algebra has a non-zero idempotent.
\src{Proved when $\kk$ has a cubic extension (Theorem~\ref{thm:fields}(2)). The natural idea of
realising a given odd-degree extension $L/\kk$ as the residue field of the scheme $Z'$ of
non-zero idempotents requires controlling all $2^n-1$ points simultaneously, and we could not
do this already for extensions of degree $5$ without cubic subextensions.}
\end{conjecture}

\begin{conjecture}[Generic Galois group of idempotents]\label{c:galois}
For the generic $n$-dimensional evolution algebra over $\mathbb Q$ (structure constants
independent indeterminates), the Galois group of the $2^n-1$ non-zero idempotents is the full
symmetric group $S_{2^n-1}$. Consequently, by Chebotarev's theorem, for fixed $n$ and
$q\to\infty$ the number of non-zero idempotents of a random perfect evolution
$\mathbb F_q$-algebra is distributed like the number of fixed points of a random permutation of
$2^n-1$ letters; in particular the probability of having no non-zero idempotent tends to
$\sum_{k=0}^{2^n-1}(-1)^k/k!$.
\src{Proved for $n=2$ (the example after Theorem~\ref{thm:fields} specialises the generic cubic
to one with Galois group $S_3$). The mean value $1$ agrees with Remark~\ref{rem:average} for
every $q$. Computer experiments with random matrices over $\mathbb F_p$ for $n=2$
($p=101,211$) and $n=3$ ($p=31,53$) give distributions of the number of non-zero idempotents
close to the fixed-point distribution of $S_3$ and $S_7$. Showing that the monodromy group is
$2$-transitive and contains a transposition seems to require a degeneration argument that we
could not complete for $n=3$.}
\end{conjecture}

\begin{conjecture}[Largest automorphism groups of perfect algebras]\label{c:maxperfect}
Over $\mathbb C$, the maximal order of the automorphism group of a perfect $n$-dimensional
evolution algebra is
\[
\max\Bigl\{\prod_{l}\bigl(l(2^l-1)\bigr)^{m_l}m_l! \;:\; \sum_l lm_l=n\Bigr\},
\]
that is $1,6,21,72,155,1296$ for $n=1,\dots,6$ and $n!$ for $n\ge7$.
\src{The values are attained by direct sums of cyclic algebras $e_i^2=e_{i+1}$ (the automorphism
group of a direct sum of $m$ copies of a perfect indecomposable algebra $C$ is
$\Aut(C)\wr S_m$). The upper bound was checked for $n\le4$ by a computer enumeration using the
bound $|\Aut(A)|\le|\mathbf D(\mathbb C)|\cdot|\Aut(\Gamma)|$ over all digraphs $\Gamma$ that
are graphs of perfect algebras. Remark~\ref{prop:diag} bounds $|\mathbf D|$ but the
interplay between $\mathbf D$ and the permutation part is what we could not control in
general.}
\end{conjecture}

\begin{conjecture}[Largest automorphism groups of simple algebras]\label{c:maxsimple}
Over $\mathbb C$, the maximal order of the automorphism group of a simple $n$-dimensional
evolution algebra is $\max\{n!,\;n(2^n-1)\}$.
\src{Both values are attained: $n(2^n-1)$ by the cyclic algebra $e_i^2=e_{i+1}$, and $n!$ by
$e_i^2=ae_i+b\sum_{j\ne i}e_j$ for generic $a,b$ (complete graph with loops, so
$\mathbf D=1$). Verified for $n\le4$ by the same computer enumeration as in
Conjecture~\ref{c:maxperfect}. The complex $n$-dimensional evolution algebras with
automorphism group $S_n$ are classified in \cite{SZ22}.}
\end{conjecture}

\begin{conjecture}[Semisimple derivation algebras]\label{c:so}
Let $\kk$ be algebraically closed of characteristic $0$. If $\Der(A)$ is semisimple for a
finite-dimensional evolution algebra $A$, then $\Der(A)\cong\bigoplus_i\mathfrak{so}_{m_i}$ with
$m_i\ge3$.
\src{Proposition~\ref{prop:On} realises $\mathfrak{so}_m$. Heuristically, derivations of an
algebra with a natural basis must preserve a family of quadratic forms simultaneously
diagonalised by that basis, and the orthogonal algebras are the natural semisimple symmetry
algebras of such configurations. We know no example in which a simple Lie algebra not
isomorphic to some $\mathfrak{so}_m$ (for instance $\mathfrak{sl}_3$ or
$\mathfrak{g}_2$) occurs as an ideal of a derivation algebra, but we could not exclude, for
instance, $\mathfrak{sl}_3$ acting on a non-degenerate algebra whose square has dimension
greater than one.}
\end{conjecture}

\begin{conjecture}[Numbers of complex idempotents]\label{c:uniqueidemp}
For every $n$ and every $m\in\{2,\dots,2^n\}$ there is a perfect $n$-dimensional complex
evolution algebra with exactly $m$ distinct idempotents (including $0$). In particular,
for every $n$ there is a perfect complex evolution algebra with a unique non-zero idempotent.
\src{True for $n\le2$: $e_1^2=\frac34e_1+\frac14e_2$, $e_2^2=\frac14e_1+\frac34e_2$ has the unique
non-zero idempotent $e_1+e_2$ (of multiplicity $3$), and triangular algebras as in the proof of
Theorem~\ref{thm:realidemp} give $3$ and $4$. Over $\mathbb R$ the analogous statement is
Theorem~\ref{thm:realidemp}, but there non-real solutions may be ignored. Triangular
constructions only give values $\ge n+1$, and the symmetric families we tried for $n=3$ (with
$S_3$-symmetry) did not produce a unique non-zero idempotent.}
\end{conjecture}

%\begin{conjecture}[Universality of simple algebras; proved]\label{c:simpleall}
% Proved: see Theorem~\ref{thm:simpleuniv}.

%\end{conjecture}

%\begin{conjecture}[Growth function over finite fields; proved]\label{c:growthFq}
% Proved: see
%Corollary~\ref{cor:growthall}.

%\end{conjecture}

\begin{conjecture}[Failure of the Hasse principle]\label{c:hasse}
There exist perfect evolution $\mathbb Q$-algebras $A,B$ which are isomorphic over $\mathbb R$
and over $\mathbb Q_p$ for every prime $p$, but not over $\mathbb Q$.
\src{Forms of a perfect algebra with the same graph are governed by
$H^1(\mathbb Q,\AutS(A))$. For two-dimensional cyclic algebras the relevant set is
$\mathbb Q^\times/(\mathbb Q^\times)^3$ (up to the action of the permutation part), and the
Hasse principle holds there; failures should come from non-abelian or non-cyclic diagonal
parts, where Tate--Shafarevich sets of finite group schemes are known to be non-trivial in
general. We could not yet produce an explicit pair.}
\end{conjecture}

\begin{conjecture}[PORC counting]\label{c:porc}
For every $n$, the number $N_n(q)$ of isomorphism classes of perfect $n$-dimensional evolution
$\mathbb F_q$-algebras is a PORC function of $q$ (a polynomial on each residue class of $q$
modulo some integer).
\src{The count is given by Burnside's lemma for $S_n\ltimes(\mathbb F_q^\times)^n$ acting on
$\GL_n(\mathbb F_q)$, and each fixed-point set is defined by monomial conditions; the
difficulty is to show that the number of invertible matrices in each such set is PORC.
Computer counts: $N_2(q)=\frac{q^2+2q-1}{2}+[\,q\equiv1\bmod3\,]$ for all odd primes $3\le
q\le23$, and $N_3(3)=249$, $N_3(5)=3943$, $N_3(7)=26249$. Explicit formulas for $n\le4$ are
given in \cite{WZ23}.}
\end{conjecture}

\begin{conjecture}[Spectrum of nilpotency indices]\label{c:nilindex}
Let $\kk$ be an infinite field. An $n$-dimensional nilpotent evolution $\kk$-algebra of
nilpotency index $k+1$ exists if and only if either $k=2^a$ with $0\le a\le n-1$, or $k$ has
at least two non-zero binary digits and $\lfloor\log_2k\rfloor+s_2(k)\le n-1$, where $s_2(k)$ is
the number of non-zero binary digits of $k$.
\src{The case of maximal index $2^{n-1}+1$ is known \cite{CGOT13a,CLOR13}. The predicted sets of
indices are $\{2,3,5\}$, $\{2,3,4,5,9\}$, $\{2,\dots,7,9,17\}$ and
$\{2,\dots,11,13,17,33\}$ for $n=3,4,5,6$. In computer experiments (exact computation modulo
a large prime with structure matrices having small integer entries; exhaustive over a fixed
set of entries for $n\le5$ and random for $n=6$) exactly these indices were found. The
lower series of an evolution algebra is governed by sums of path weights in the graph, and
binary expansions appear because $e_i^2$ doubles the degree; turning this into a proof of
the ``only if'' part requires controlling cancellations, which we could not do.}
\end{conjecture}

\end{document}